\documentclass[smallextended,referee,envcountsect]{svjour3}
\smartqed
\usepackage[doublespacing]{setspace}
\usepackage{graphicx}
\usepackage{amssymb}
\usepackage{amsmath}
\usepackage{amsfonts}
\usepackage{graphics,graphicx}
\usepackage[pdftex,pdfstartview=FitH]{hyperref}

\begin{document}

\begin{singlespace}

\title{\textsf{Calculating Radius of Robust Feasibility of Uncertain Linear
Conic Programs via Semidefinite Programs}}


\author{M.A. Goberna  \and  V. Jeyakumar \and G. Li %
}

\institute{M.A. Goberna, Corresponding author \at
University of Alicante, 03080 Alicante, Spain, Email: mgoberna@ua.es
           \and
           V. Jeyakumar    \at
              University of New South Wales, Sydney, NSW 2052, Email: v.Jeyakumar@unsw.edu.au.
              \and
               G. Li   \at
              University of New South Wales, Sydney, NSW 2052, Email: g.li@unsw.edu.au
}

\date{Received: date / Accepted: date}

\maketitle

\begin{abstract}
The radius of robust feasibility provides a numerical value for the largest
possible uncertainty set that guarantees robust feasibility of an uncertain
linear conic program. This determines when the robust feasible set is
non-empty. Otherwise the robust counterpart of an uncertain program is not
well-defined as a robust optimization problem. In this paper, we address a
key fundamental question of robust optimization: How to compute the radius
of robust feasibility of uncertain linear conic programs, including linear
programs? We first provide computable lower and upper bounds for the radius
of robust feasibility for general uncertain linear conic programs under the
commonly used ball uncertainty set. We then provide important classes of
linear conic programs where the bounds are calculated by finding the optimal
values of related semidefinite linear programs, among them uncertain
semidefinite programs, uncertain second-order cone programs and uncertain support vector
machine problems. In the case of an uncertain linear program, the exact
formula allows us to calculate the radius by finding the optimal value of an
associated second-order cone program.
\end{abstract}


\section{Introduction}

Consider the following linear conic programming problem
\begin{equation}
\begin{array}{ll}
(P)\  & \ \displaystyle \min_{x\in \mathbb{R}^{n}}{c}^{T}x \medskip \\
& \text{s.t. }\left[
\begin{array}{c}
a_{1}^{T}x+b_{1} \\
\vdots \\
a_{m}^{T}x+b_{m}%
\end{array}%
\right] \in -K,%
\end{array}%
\end{equation}%
where $\left\{ 0_{m}\right\} \neq K\varsubsetneqq \mathbb{R}^{m}$ is a given
closed pointed convex cone with nonempty interior (which implies that its
positive dual cone $K^{\ast }$ enjoys the same properties), ${c}\in \mathbb{R%
}^{n}$, and the vectors $\left( a_{i},b_{i}\right) \in \mathbb{R}%
^{n+1},1\leq i\leq m,$ or, equivalently, the matrix $A:=\left[ a_{1}\mid
...\mid a_{m}\right] ^{T}\in \mathbb{R}^{m\times n}$ and the vector $%
b:=\left( b_{1},...,b_{m}\right) ^{T}\in \mathbb{R}^{m},$ are the data
associated to the conic programming problem $(P)$. This model problem has
found numerous applications in engineering, statistics and finance (\cite%
{Anjos17}, \cite{Boyd}), and covers many important optimization problems
such as

\begin{description}
\item[\textrm{(SDPs)}] Semi-definite programming (SDP in brief) problems,
where $K=S_{+}^{q}$ is the cone consisting of all $(q\times q)$ positive
semi-definite symmetric matrices.

\item[\textrm{(SOCPs)}] Second order cone programming problems, where
\begin{equation*}
K=\left\{ x\in \mathbb{R}^{m}:x_{m}\geq \left\Vert \left(
x_{1},...,x_{m-1}\right) \right\Vert \right\}
\end{equation*}%
is the so-called second order cone (SOC in short), usually denoted by $%
K_{p}^{m}$.

\item[\textrm{(LPs)}] Linear programming (LP) problems, where $K=\mathbb{R}%
_{+}^{m}$. We must emphasize that most results in this paper are new even in
this particular setting.
\end{description}

In practice, the data associated to the optimization problem $(P)$ are
uncertain due to measurement errors or prediction errors. One of the
prominent ways of dealing with optimization under data uncertainty is the
robust optimization approach. Following this approach, one assumes that the
constraint data $\left( a_{i},b_{i}\right) $ ranges in some \emph{%
uncertainty set}. We note that it is known in the robust optimization
literature \cite{BEL09} that the general case where the linear objective
data $c$ is also uncertain can be easily converted to the current setting by
introducing an auxiliary variable.
\medskip

\noindent(\textbf{Uncertainty sets})
Given $i\in \left\{ 1,...,m\right\} ,$\ the uncertainty set for the
constraint data $\left( a_{i},b_{i}\right) $ is the ball $\mathcal{U}%
_{r_{i}}:=\left( \overline{a}_{i},\overline{b}_{i}\right) +r_{i}\mathbb{B}%
_{n+1},$ where $r_{i}\in \mathbb{R}_{+}$ (with $r_{i}=0$ when $\left(
a_{i},b_{i}\right) $ is deterministic) and $\mathbb{B}_{n+1}$ is the
Euclidean closed unit ball in $\mathbb{R}^{n+1}.$

Denote $\overline{A}=\left[ \overline{a}_{1}\mid ...\mid \overline{a}_{m}%
\right] ^{T}$ and $\overline{b}=\left( \overline{b}_{1},...,\overline{b}%
_{m}\right) ^{T}.$ In practical situations, the matrix $\left[ \overline{A},%
\overline{b}\right] $ is called the nominal data which may correspond to a
central value of a sample of observed matrices $\left[ A,b\right] $. For
instance, $\left( \overline{a}_{i},\overline{b}_{i}\right) $ could be the
mean vector of a sample of $\left( a_{i},b_{i}\right) -$vectors. Denoting $%
r=(r_{1},\ldots ,r_{m})\in \mathbb{R}_{+}^{m},$\ the \emph{robust counterpart%
} of $(P),$ depending on $r,$ can be formulated as%
\begin{equation*}
\begin{array}{ll}
(P_{r}) & \min_{x\in \mathbb{R}^{n}}{c}^{T}x\medskip \\
& \text{s.t. }\left[
\begin{array}{c}
a_{1}^{T}x+b_{1} \\
\vdots \\
a_{m}^{T}x+b_{m}%
\end{array}%
\right] \in -K,\forall (a_{i},b_{i})\in \mathcal{U}_{r_{i}},1\leq i\leq m,%
\end{array}%
\end{equation*}%
which gives the worst-case solution for all possible realization of the
scenarios in the constraint uncertainty set. The feasible set of $(P_{r})$
\begin{equation*}
F_{r}(\overline{A},\overline{b}):=\left\{ x\in \mathbb{R}^{n}:Ax+b\in -K,%
\text{ }\forall \left( a_{i},b_{i}\right) \in \mathcal{U}_{r_{i}},i=1,...,m%
\right\}
\end{equation*}%
is referred as the robust feasible set, where $A:=\left[ a_{1}\mid ...\mid
a_{m}\right] ^{T}\in \mathbb{R}^{m\times n}$ and $b:=\left(
b_{1},...,b_{m}\right) ^{T}\in \mathbb{R}^{m}$.

A great deal of work has been published on the efficient computation of
optimal solutions of the robust counterpart of a variety of uncertain
optimization problems. These problems are inherently semi-infinite, but they
can be reduced to tractable semi-definite programs under suitable
assumptions (see, e.g., \cite{BEL09}, \cite{ben-nem3}, \cite{BN01}, \cite{BN}%
, and references therein).

One of the key questions in robust optimization is to determine when the
robust feasible set is indeed nonempty (otherwise, the robust counterpart
problem is not well-defined). An important concept quantifying the robust
feasibility issue is the so-called \emph{radius of robust feasibility} (RRF
in short) of the nominal data, that can be roughly defined as the largest
size of the uncertainty set so that the robust feasible set is nonempty.
Formulas for the RRF have been given in \cite{CJ16}, \cite{GJLP15} and \cite%
{LST20} for robust LP problems, in \cite{GJLP} for robust linear
semi-infinite programming (LSIP) problems, in \cite{CLLLY20}, \cite{GJLL16}
and \cite{LW18} for robust convex programs, and in \cite{LST20} for robust
mixed-integer LP problems. The uncertainty sets are balls in \cite{GJLP} and
\cite{GJLP15}, spectrahedra (e.g., ellipsoids, polytopes, and boxes) in \cite%
{CJ16}, and more general compact convex sets in \cite{CLLLY20}, \cite{GJLL16}%
, \cite{LW18} and \cite{LST20}. Most of the mentioned works discuss
tractability issues regarding the proposed formulas for the computation of
the RRF. In particular, in \cite{CJ16} and \cite{GJLL16} the RRF of linear
and convex polynomial problems is computed by solving associated
semi-definite programs while \cite{LST20} proposes to compute the RRF of
linear and mixed linear programs via fractional programming and effective
binary search algorithms, respectively. The introduction of the latter paper
briefly reviews applications of the RRF to facility location design \cite%
{CN03}, flexibility index problem \cite{ZGL16}, and design and control of
gas networks \cite{KHPS15}.

The uncertainty sets are also Euclidean balls in \cite{GJLP15}, devoted to
certify the existence of highly robust solutions (i.e., robust feasible
solutions which are optimal for any scenario) in robust multi-objective
linear and convex programming with uncertain objectives through the
estimation of the corresponding radius of highly robust efficiency. The
formulas provided in this paper can also be used in robust scalar (resp.,
multi-objective) linear conic programming with deterministic constraints and
uncertain objective (resp., objectives).

Although any conic LP problem can be reformulated as an LSIP problem through
the \textquotedblleft natural\textquotedblright\ linearization of the conic
linear system $\left\{ \overline{A}x+\overline{b}\in -K\right\} $ as $\sigma
_{K^{\ast }}:=\left\{ a^{T}x\leq -b,\left( a,b\right) \in T_{K^{\ast
}}\right\} ,$ where the index set is $$T_{K^{\ast }}:=\left\{ \
\sum\limits_{i=1}^{m}\lambda _{i}(\overline{a}_{i},\overline{b}_{i}):\lambda
\in K^{\ast }\right\} ,$$ with $K^{\ast }$ denoting the dual cone of $K,$
this approach is seldom useful in examining uncertain conic LP, in
particular from the stability and robustness perspectives, as $\sigma
_{K^{\ast }}$ is not even stable in the sense of LSIP because the latter
system contains the trivial inequality $0_{m}^{T}x\leq 0$ (recall \cite[%
Theorem 6.1]{GL98}). Therefore, none of the approaches used in the above
mentioned papers \cite{CJ16,GJLP15,LST20,GJLP,CLLLY20,GJLL16,ZGL16} can be
directly adapted to deal with problems with conical constraints which calls
for further research on the study of RRF for uncertain conic programming
problems. Moreover, the mathematical formulae for estimating the RRF are
often very difficult to validate numerically. In this paper we make the following contributions to robust linear conic
programming:

\begin{description}
\item[(i)] We first establish computable lower and upper bounds for the radius
of robust feasibility for general uncertain linear conic programs under the
commonly used ball uncertainty set.

\item[(ii)] We then show how the bounds can be calculated for important classes of linear conic programs by finding the optimal values of related
semidefinite linear programs (SDPs), among them uncertain SDPs, uncertain
second-order cone programs and uncertain support vector machine problems. In
the case of an uncertain linear program, the exact formula allows us to
calculate the radius by finding the optimal value of an associated
second-order cone program.
\end{description}


The paper is organized as follows. Section 2 introduces the important
concepts of radius of robust feasibility as well as admissible set of
parameters formed by the parameters $r\in \mathbb{R}_{+}^{m}$ such that the
robust feasible set $F_{r}(\overline{A},\overline{b})$ is nonempty. Section
3 provides lower and upper bounds for the RRF of uncertain conic programs.
Section 4 establishes computational tractable bounds for uncertain
semi-definite programs and second order cone programs. Section 5 provides  conclusions and some future research directions. Finally, the appendix
presents proofs of certain technical results.

\section{Preliminaries and Radius of Robust Feasibility}

Let us start introducing the necessary notation. We denote by $0_{n}$, $%
1_{n},$ $\left\Vert \cdot \right\Vert ,$ $\mathbb{B}_{n},$ and $d$ the
vector of all zeros, the vector of all ones, the Euclidean norm, the
corresponding closed unit ball, and the Euclidean distance in $\mathbb{R}%
^{n} $, respectively. We also denote by $\left\{ e_{1},...,e_{n}\right\} $
the canonical basis and by $\Delta _{n}=\left\{ x\in \mathbb{R}%
_{+}^{n}:1_{n}^{T}x=1\right\} $ the unit simplex in $\mathbb{R}^{n}.$ Given $%
\emptyset \neq X\subseteq \mathbb{R}^{n}$, $\mathop{\rm int}X$, $%
\mathop{\rm
bd}X$, $\overline{X}$, $\mathop{\rm aff}X,$ $\mathop{\rm conv}X,$ denote the
interior, the boundary, the closure, the affine hull, and the convex hull of
$X$, respectively, whereas $\mathop{\rm cone}X:=\mathbb{R}_{+}%
\mathop{\rm
conv}X$ denotes the convex conical hull of $X\cup \{0_{n}\}$. We represent
by $\mathop{\rm dist}\left( \overline{x},X\right) =\inf_{x\in X}d\left(
\overline{x},x\right) $ the distance from $\overline{x}$ to a set $%
X\subseteq \mathbb{R}^{n},$ with $\mathop{\rm dist}(\overline{x},\emptyset
)=+\infty $ by convention. A set $K$ is called a cone if and only if  $\lambda x\in K$
for any $\lambda \geq 0$ and $x\in K$. The (positive) dual cone of a cone $%
K\subseteq \mathbb{R}^{m}$ is defined as $$K^{\ast }:=\{a\in \mathbb{R}%
^{m}:a^{T}x\geq 0\mbox{
for all }x\in K\}.$$

Throughout the paper, we assume that the cone $\left\{ 0_{m}\right\} \neq
K\varsubsetneqq \mathbb{R}^{m}$ is a given closed pointed convex cone with
nonempty interior. Moreover, we also assume that the feasible set of the
nominal problem $P_{0_{m}}$ is nonempty, that is, $\left\{ x\in \mathbb{R}%
^{n}:\overline{A}x+\overline{b}\in -K\right\} \neq \emptyset $, where $%
\overline{A}=\left[ \overline{a}_{1}\mid ...\mid \overline{a}_{m}\right]
^{T} $ and $\overline{b}=\left( \overline{b}_{1},...,\overline{b}_{m}\right)
^{T}. $  We note that these assumptions on the cone $K$ in $\left( P\right) $\ are
standard assumptions in the linear conic programming literature. In
particular, the assumption $\left\{ 0_{m}\right\} \neq K\varsubsetneqq
\mathbb{R}^{m}$ eliminates uninteresting cases (the feasible sets of the
involved problems being either affine manifolds or the whole space); the
condition that $\mathop{\rm int}K\neq \emptyset $ ensures \cite{GRTZ03,Beer93} the
existence of a compact base $\mathcal{B}$ for $K^{\ast }$ (i.e., a compact
and convex subset $\mathcal{B}$ of $K^{\ast }$ such that $0_{m}\notin
\mathcal{B}$ and $K^{\ast }=\mathbb{R}_{+}\mathcal{B}$).

Next, we introduce the important definition of the admissible set of
parameters, which is formed by those parameters $r\in \mathbb{R}_{+}^{m}$ so
that the robust feasible set is nonempty. {This definition plays an
important role in defining the concept of RRF considered later. We also
emphasize that this concept appears to be new and was not examined in the
previous study of the literature of RRF.}

\begin{definition}[\textbf{Admissible set}]
Let $\overline{A}=\left[ \overline{a}_{1}\mid ...\mid \overline{a}_{m}\right]
^{T}\in \mathbb{R}^{m\times n}$ and let $\overline{b}=\left( \overline{b}%
_{1},...,\overline{b}_{m}\right) ^{T}\in \mathbb{R}^{m}.$ The set
\begin{equation*}
C(\overline{A},\overline{b}):=\{r=(r_{1},\ldots ,r_{m})\in \mathbb{R}%
_{+}^{m}:F_{r}(\overline{A},\overline{b})\neq \emptyset \}\subseteq \mathbb{R%
}^{m}
\end{equation*}%
is called the \emph{admissible set of parameters }of the uncertain problem $%
\left( P\right) $, where $F_{r}(\overline{A},\overline{b})$ is the robust
feasible set given by
\begin{equation*}
F_{r}(\overline{A},\overline{b}):=\left\{ x\in \mathbb{R}^{n}:Ax+b\in -K,%
\text{ }\forall \left( a_{i},b_{i}\right) \in \mathcal{U}_{r_{i}},i=1,...,m%
\right\} ,
\end{equation*}%
$A:=\left[ a_{1}\mid ...\mid a_{m}\right] ^{T}\in \mathbb{R}^{m\times n}$, $%
b:=\left( b_{1},...,b_{m}\right) ^{T}\in \mathbb{R}^{m}$ and $\mathcal{U}%
_{r_{i}}$ is the ball uncertainty set defined as $\mathcal{U}%
_{r_{i}}:=\left( \overline{a}_{i},\overline{b}_{i}\right) +r_{i}\mathbb{B}%
_{n+1},$
\end{definition}

From its definition, $C(\overline{A},\overline{b})$ is radiant (in the sense
that $\mu C(\overline{A},\overline{b})\subseteq C(\overline{A},\overline{b})$
for all $\mu \in \left[ 0,1\right] $). However, we observe that

\begin{itemize}
\item The set $C(\overline{A},\overline{b})$ may be reduced to $\left\{
0_{m}\right\}$. For example, $C(\overline{A},\overline{b})=\left\{
0_{2}\right\} $ when $\left[ \overline{A}\mid \overline{b}\right] =\left[
\begin{array}{cc}
1 & 0 \\
-1 & 0%
\end{array}%
\right] ,$ with $K=-\mathbb{R}_{+}^{2}$.

\item The set $C(\overline{A},\overline{b})$ may contain nonzero points
despite being contained in $\mathop{\rm bd}\mathbb{R}_{+}^{m}.$ For
instance, $C(\overline{A},\overline{b})=\mathop{\rm conv}\left\{ \left(
0,0,0\right) ,\left( 0,0,1\right) \right\} $ when $\left[ \overline{A}\mid
\overline{b}\right] =\left[
\begin{array}{cc}
1 & 0 \\
-1 & 0 \\
0 & -1%
\end{array}%
\right] $ and $K=-\mathbb{R}_{+}^{3}$.

\item The set $C(\overline{A},\overline{b})$ can be a non-closed set with
nonempty interior. For example, $C(\overline{A},\overline{b})=\left[
0,1\right) ^{2}$ when $\left[ \overline{A}\mid \overline{b}\right] =\left[
\begin{array}{cc}
1 & 1 \\
1 & 2%
\end{array}%
\right] $ and $K=-\mathbb{R}_{+}^{2}$.
\end{itemize}

Next, we summarize some basic properties of the admissible set of parameters
below. These interesting mathematical properties could be of some
independent interest. Its proof is included in the appendix for the purpose
of self-containment.

\begin{proposition}[\textbf{Basic properties of }$C(\overline{A},\overline{b}%
)$]
\label{Prop_Consist_Set3} Consider the admissible set $C(\overline{A},%
\overline{b})$.

\begin{itemize}
\item[\textrm{(a)}] (Dual characterization)
\begin{equation*}
C(\overline{A},\overline{b})=\left\{ r\in \mathbb{R}_{+}^{m}:\left(
0_{n},1\right) \notin \overline{\mathop{\rm cone}\displaystyle%
\bigcup\limits_{\lambda \in \mathcal{B}}\left\{ \displaystyle%
\sum\limits_{i=1}^{m}\lambda _{i}\left( (\overline{a}_{i},\overline{b}%
_{i})+r_{i}\mathbb{B}_{n+1}\right) \right\} }\right\} .\newline
\end{equation*}
where $\mathcal{B}$ is a compact base of $K^{\ast }$.

\item[\textrm{(b)}] (Boundedness) The set $C(\overline{A},\overline{b})$ is
a bounded set which can be expressed as a union of segments emanating from $%
0_{m}$.

\item[\textrm{(c)}] (Characterization of the interior) Let $\sigma _{r}^{%
\mathcal{B}}$ be the linear system describing $F_{r}(\overline{A},\overline{b%
})$ in (\ref{2.2}). If $r\in \mathbb{R}_{++}^{m}$ and $\sigma _{r}^{\mathcal{%
B}}$ satisfies the Slater condition, then $r\in \mathop{\rm int}C(\overline{A%
},\overline{b}).$ Conversely, if $r\in \mathop{\rm int}C(\overline{A},%
\overline{b})$ and $\sum\limits_{i=1}^{m}r_{i}\lambda _{i}>0$ for all $%
\lambda \in \mathcal{B},$ then $\sigma _{r}^{\mathcal{B}}$ satisfies the
Slater condition.
\end{itemize}
\end{proposition}

We now introduce the concept of radius of robust feasibility (\emph{RRF}) for $\left( P\right) ,$ as the supremum
of those $\alpha \geq 0$ such that $F_{\left( \alpha ,...,\alpha \right) }(%
\overline{A},\overline{b})\neq \emptyset .$

\begin{definition}[\textbf{Radius of robust feasibility}]
\label{Def_radius}The radius of robust feasibility (\emph{RRF}) for $\left( P\right) $ is defined as
\begin{equation*}
\rho (\overline{A},\overline{b}):=\sup \{\alpha \in \mathbb{R}_{+}:\alpha
1_{m}\in C(\overline{A},\overline{b})\},  \label{2.1}
\end{equation*}%
where $1_{m}$ is the vector in $\mathbb{R}^{m}$ whose components are all
equal to one.
\end{definition}

The RRF $\rho (\overline{A},\overline{b})$ (Definition \ref{Def_radius}) is,
roughly speaking, a measure of the maximal perturbations of a robust
optimization problem which result in the problem still being feasible. The
precise definition is a natural extension to robust conic linear programs of
the homonymous concept introduced in \cite{GJLP15} and \cite{GJLP} in the
robust LP and LSIP settings, respectively.

%
%

\begin{proposition}[\textbf{Basic properties of radius of robust feasibility}%
]
\label{Prop_Ro}The following properties of the RRF of (P) hold:\newline
(i) $\rho (\overline{A},\overline{b})=\sup\nolimits_{r\in C(\overline{A},%
\overline{b})}\min\nolimits_{1\leq i\leq m}r_{i}.$ \newline
(ii) If $\rho (\overline{A},\overline{b})>0,\ \left[ 0,\rho (\overline{A},%
\overline{b})\right) ^{m}\subseteq C(\overline{A},\overline{b}).$\newline
(iii) $\rho (\overline{A},\overline{b})>0$ if and only if $\mathop{\rm int}C(%
\overline{A},\overline{b})\neq \emptyset .$\newline
(iv) If there exists $\bar{x}\in \mathbb{R}^{n}$ such that $\overline{A}%
\overline{x}+\overline{b}\in -\mathop{\rm int}K$, then $\rho (\overline{A},%
\overline{b})>0.$
\end{proposition}

\begin{proof}
We first see that the admissible set of parameters $C(\overline{A},\overline{%
b})$ satisfies
\begin{equation}  \label{eq:useful0}
C(\overline{A},\overline{b})=\left( C(\overline{A},\overline{b})-\mathbb{R}%
_{+}^{m}\right) \cap \mathbb{R}_{+}^{m}.
\end{equation}
If $r\in C(\overline{A},\overline{b}),$ by definition of $C(\overline{A},%
\overline{b}),$ any vector of $\left( r-\mathbb{R}_{+}^{m}\right) \cap
\mathbb{R}_{+}^{m}$ belongs to $C(\overline{A},\overline{b}).$ So, $\left( C(%
\overline{A},\overline{b})-\mathbb{R}_{+}^{m}\right) \cap \mathbb{R}%
_{+}^{m}\subseteq C(\overline{A},\overline{b}).$ Conversely, since
\begin{equation*}
C(\overline{A},\overline{b})=\left( C(\overline{A},\overline{b}%
)-0_{m}\right) \cap \mathbb{R}_{+}^{m}\subseteq \left( C(\overline{A},%
\overline{b})-\mathbb{R}_{+}^{m}\right) \cap \mathbb{R}_{+}^{m},
\end{equation*}%
we have $C(\overline{A},\overline{b})=\left( C(\overline{A},\overline{b})-%
\mathbb{R}_{+}^{m}\right) \cap \mathbb{R}_{+}^{m}.$

(i) Let $\alpha \geq 0$ be such that $\alpha 1_{m}\in C(\overline{A},%
\overline{b}).$ Since $\alpha =\min \left\{ \alpha ,...,\alpha \right\} \leq
\sup\nolimits_{r\in C(\overline{A},\overline{b})}\min\nolimits_{1\leq i\leq
m}r_{i}$, we have $\rho (\overline{A},\overline{b})\leq \sup\nolimits_{r\in
C(\overline{A},\overline{b})}\min\nolimits_{1\leq i\leq m}r_{i}.$

Conversely, given $\epsilon >0,$ there exists $r\in C(\overline{A},\overline{%
b})$ and $j\in \left\{ 1,...,m\right\} $ such that $r_{j}\leq r_{i}$ for all
$1\leq i\leq m$ and $\sup\nolimits_{r\in C(\overline{A},\overline{b}%
)}\min\nolimits_{1\leq i\leq m}r_{i}-\epsilon <r_{j}.$ Since $%
r-r_{j}1_{m}\in \mathbb{R}_{+}^{m},$ \eqref{eq:useful0} yields
\begin{equation*}
r_{j}1_{m}\in \left( r-\mathbb{R}_{+}^{m}\right) \cap \mathbb{R}%
_{+}^{m}\subseteq C(\overline{A},\overline{b}).
\end{equation*}%
Then, $r_{j}\leq \rho (\overline{A},\overline{b}),$ so that $%
\sup\nolimits_{r\in C(\overline{A},\overline{b})}\min\nolimits_{1\leq i\leq
m}r_{i}-\epsilon <\rho (\overline{A},\overline{b})$ for all $\epsilon >0.$
So, we see that $\sup\nolimits_{r\in C(\overline{A},\overline{b}%
)}\min\nolimits_{1\leq i\leq m}r_{i}\leq \rho (\overline{A},\overline{b}).$

(ii) It follows from \eqref{eq:useful0}.

(iii) The direct statement follows from (ii) and the converse from (i).

(iv) The mapping $\Phi \left( A,b\right) :=A\overline{x}+b$ is continuous on
$\mathbb{R}^{m\times n}\times \mathbb{R}^{m}$ and satisfies $\Phi \left(
\overline{A},\overline{b}\right) =\overline{A}\bar{x}+\overline{b}\in -%
\mathop{\rm int}K$, so that $\Phi \left( A,b\right) \in -K$ for $\left(
A,b\right) $ close enough to $\left( \overline{A},\overline{b}\right) .$
Thus, $\overline{x}\in F_{r}(\overline{A},\overline{b})$ for any $r\in
\mathbb{R}_{+}^{m}$ sufficiently close to $0_{m}$ and, so, $\rho (\overline{A%
},\overline{b})=\sup\nolimits_{r\in C(\overline{A},\overline{b}%
)}\min\nolimits_{1\leq i\leq m}r_{i}>0$.$\medskip $
\end{proof}

The characterization of $\rho (\overline{A},\overline{b})$ in Proposition %
\ref{Prop_Ro}(i)\ was used in\ \cite{GJLL16}\ as definition of RRF for\ a
class of convex problems with polynomial constraints without proving the
equivalence between both concepts. Statement (ii)\ shows the identifiable
part of $C(\overline{A},\overline{b})$ when $\rho (\overline{A},\overline{b}%
) $ can be computed (see Theorem \ref{TheorComput} below): the (non-closed)
hypercube $\left[ 0,\rho (\overline{A},\overline{b})\right) ^{m}.$ The
interior point condition $\overline{A}\overline{x}+\overline{b}\in -%
\mathop{\rm int}K$\ in statement (iv) is called \emph{Slater condition} for $%
\left( \overline{P}\right) .$


%

Next, we present a simple linear program example illustrating the admissible
set of parameters and the radius of robust feasibility. This example is
mainly used for illustrative purposes. Examples of conic program will also
be proposed based on this illustrative example and discussed later on.

\begin{example}[\textbf{Illustrative example}]
\label{example1} \label{Example1} Consider the simple linear program
\begin{equation}
\min_{x\in \mathbb{R}}\ x\medskip ,\text{ s.t. }\left[
\begin{array}{c}
2x \\
-x-3%
\end{array}%
\right] \in -\mathbb{R}_{+}^{2}.  \label{2.4}
\end{equation}%
Proposition \ref{Prop_Consist_Set3} part (a) will allow us to obtain a
tomographic description of the plane set
\begin{equation*}
C(\overline{A},\overline{b})=\left\{ r\in \mathbb{R}_{+}^{2}:\left(
0,1\right) \notin \overline{\mathop{\rm cone}\left\{ \left( \left(
2,0\right) +r_{1}\mathbb{B}_{2}\right) \cup \left( \left( -1,-3\right) +r_{2}%
\mathbb{B}_{2}\right) \right\} }\right\}
\end{equation*}%
as the union of its intersections with all vertical lines. \newline
Given $r=\left( r_{1},r_{2}\right) \in \mathbb{R}^{2}$ such that $r_{1}\geq
2,\ F_{r}(\overline{A},\overline{b})=\emptyset .$ So, we fix $0\leq r_{1}<2$
and calculate those $r_{2}$ such that $F_{r}(\overline{A},\overline{b})\neq
\emptyset .$ For $r_{1}=0$, $F_{r}(\overline{A},\overline{b})\neq \emptyset $
if and only if $0\leq r_{2}\leq d\left( \left( -1,3\right) ,y=0\right) =3.$
We now take $0<r_{1}<2.$ The tangent lines from $0_{2}$ to $\mathop{\rm bd}%
\left( \left( 2,0\right) +r_{1}\mathbb{B}_{2}\right) $ are $\pm r_{1}x-\sqrt{%
4-r_{1}^{2}}y=0,$ denoted by $L_{+}$ and $L_{-},$ respectively. So, $F_{r}(%
\overline{A},\overline{b})\neq \emptyset $ if and only if $\left(
-1,3\right) $ is above $L_{+}$ and $L_{-},\ $i.e.,$\ r_{1}\leq \frac{6}{%
\sqrt{10}}\ $and
\begin{equation*}
r_{2}\leq \min \left\{ d\left( \left( -1,3\right) ,L_{+}\right) ,d\left(
\left( -1,3\right) ,L_{-}\right) \right\} =d\left( \left( -1,3\right)
,L_{-}\right) =\frac{3\sqrt{4-r_{1}^{2}}-r_{1}}{2}.
\end{equation*}%
This amounts to saying that%
\begin{equation*}
\begin{array}{ll}
C(\overline{A},\overline{b}) & =\left\{ r\in \mathbb{R}_{+}^{2}:0\leq
r_{1}\leq 3\sqrt{\frac{2}{5}},0\leq r_{2}\leq \frac{3\sqrt{4-r_{1}^{2}}-r_{1}%
}{2}\right\} \\
& =\left\{ r\in \mathbb{R}_{+}^{2}:5r_{1}^{2}+2r_{1}r_{2}+2r_{2}^{2}\leq
18\right\} .%
\end{array}
\label{2.9}
\end{equation*}%

For problem given in \eqref{2.4}, we illustrate how Proposition \ref%
{Prop_Ro} can help to find the greatest square supported by the coordinate
axes contained in $C(\overline{A},\overline{b}).$ As the\allowbreak\ line $%
r_{2}=r_{1}$ intersects $\mathop{\rm bd}C(\overline{A},\overline{b})$ at $%
\left( 0,0\right) $ and $r^{1}:=\left( \sqrt{2},\sqrt{2}\right) ,$ the
greatest $\alpha $ such that $F_{\left( \alpha ,\alpha \right) }\neq
\emptyset $ is $\rho (\overline{A},\overline{b})=\sqrt{2},$ with $F_{r^{1}}(%
\overline{A},\overline{b})=\left\{ \allowbreak -1\right\} $ (a singleton
set). Moreover, according to Proposition \ref{Prop_Ro}(ii), the feasibility
of the robust counterpart is guaranteed for any $r$ in the square $\left[
0,\rho (\overline{A},\overline{b})\right) ^{2}$ (possibly with $r_{2}\neq
r_{1}$).\
\end{example}

\section{Bounds for Radius of Robust Feasibility}

In this section, we obtain lower and upper bounds for $\rho (\overline{A},%
\overline{b}),$ for an arbitrary compact base $\mathcal{B}$\ of $K^{\ast },$%
\ by using the following lemma. We note that this lemma is an auxiliary
result using convex analysis, which allows to replace the finite set $C\ $%
(or, equivalently, the polytope $\mathop{\rm conv}C$) in \cite[Lemma 3]%
{GJLP15} by an arbitrary compact convex set; our proof here emphasizes the
role played by the limit superior of some sequence of scalars which arises
in the argument.

\begin{lemma}
\label{lemma:1} Let $C\subseteq \mathbb{R}^{n+1}$ be a nonempty compact
convex set and $\alpha \geq 0$. Suppose that
\begin{equation}
(0_{n},1)\in \overline{\mathop{\rm cone}\left( C+\alpha \mathbb{B}%
_{n+1}\right) }.  \label{3.1}
\end{equation}%
Then, for all $\epsilon >0$, we have
\begin{equation*}
(0_{n},1)\in \mathop{\rm cone}\big( C+(\alpha +\epsilon )\mathbb{B}_{n+1}%
\big) .
\end{equation*}
\end{lemma}

\begin{proof}
Let $\epsilon >0$. Since $(0_{n},1)\in \overline{\mathop{\rm cone}\left(
C+\alpha \mathbb{B}_{n+1}\right) }=\overline{\mathbb{R}_{+}\left( C+\alpha
\mathbb{B}_{n+1}\right) }$, there exist sequences $\{(y_{k},s_{k})\}_{k\in
\mathbb{N}}\subseteq \mathbb{R}^{n+1},$ $\{\mu _{k}\}_{k\in \mathbb{N}%
}\subseteq \mathbb{R}_{+},$ $\{(x_{k},t_{k})\}_{k\in \mathbb{N}}\subseteq C$
and $\{(z_{k},w_{k})\}_{k\in \mathbb{N}}\subseteq \mathbb{B}_{n+1}$ such
that
\begin{equation}
(y_{k},s_{k})={\mu _{k}\left( (x_{k},t_{k})+\alpha (z_{k},w_{k})\right) }%
\rightarrow (0_{n},1).  \label{3.2}
\end{equation}%
Two cases are possible for $\{\mu _{k}\}_{k\in \mathbb{N}}.$

Case 1: $\limsup {\mu _{k}<+\infty .}$ We can assume that ${\mu
_{k}\longrightarrow \mu \in }\mathbb{R}_{+}$ as $k\rightarrow \infty .$ From
(\ref{3.2}), ${\mu >0.}$ Then, for sufficiently large $k$, one has $\frac{{\mu }}{2{\mu _{k}}}<1$ and $${\mu
_{k}\left( (x_{k},t_{k})+\alpha (z_{k},w_{k})\right) \in }(0_{n},1)+\frac{%
\epsilon {\mu }}{2}\mathbb{B}_{n+1}.$$ Let
$\left( u_{k},v_{k}\right) \in \mathbb{B}_{n+1}$ be such that ${\mu
_{k}\left( (x_{k},t_{k})+\alpha (z_{k},w_{k})\right) =}(0_{n},1)+\frac{%
\epsilon {\mu }}{2}\left( u_{k},v_{k}\right) .$ Then,
\begin{equation*}
(0_{n},1)={\mu _{k}\left( (x_{k},t_{k})+\alpha (z_{k},w_{k})-\frac{\epsilon {%
\mu }}{2{\mu _{k}}}\left( u_{k},v_{k}\right) \right) }\in \mathop{\rm cone}%
\left( C+(\alpha +\epsilon )\mathbb{B}_{n+1}\right) .
\end{equation*}%
Case 2: $\limsup {\mu _{k}=+\infty .}$ We may assume that $\mu
_{k}\rightarrow +\infty $ as $k\rightarrow \infty $. We also assume by
contradiction that
\begin{equation*}
(0_{n},1)\notin \mathop{\rm cone}\left( C+(\alpha +\epsilon )\mathbb{B}%
_{n+1}\right) .
\end{equation*}%
Then, by the separation theorem, there exists $(\xi ,r)\in \mathbb{R}%
^{n+1}\backslash \{0_{n+1}\}$ such that
\begin{equation}
r=\langle (\xi ,r),(0_{n},1)\rangle \leq 0\leq \langle (\xi ,r),(y,s)\rangle
,  \label{useful}
\end{equation}%
for all $(y,s)\in \mathop{\rm cone}\left( C+(\alpha +\epsilon )\mathbb{B}%
_{n+1}\right) .$ Let $(y,s):=\frac{(\xi ,r)}{\Vert (\xi ,r)\Vert }\in
\mathbb{B}_{n+1}.$ Note that
\begin{equation*}
{\mu _{k}\left( (x_{k},t_{k})+\alpha (z_{k},w_{k})-\epsilon (y,s)\right) }%
\in \mathbb{R}_{+}\left( C+(\alpha +\epsilon )\mathbb{B}_{n+1}\right) .
\end{equation*}%
Then, (\ref{3.2}) and \eqref{useful} imply that
\begin{eqnarray*}
0 &\leq &\langle (\xi ,r),{\mu _{k}\left( (x_{k},t_{k})+\alpha
(z_{k},w_{k})\right) }\rangle -\mu _{k}\,\epsilon \Vert (\xi ,r)\Vert
\medskip \\
&=&\langle (\xi ,r),(y_{k},s_{k})\rangle -\mu _{k}\epsilon \Vert (\xi
,r)\Vert ,
\end{eqnarray*}%
with $\langle (\xi ,r),(y_{k},s_{k})\rangle \rightarrow r$ and $\mu
_{k}\epsilon \Vert (\xi ,r)\Vert \rightarrow +\infty .$ We got a
contradiction.
\end{proof}

As an illustration of the above discussion on the value of $\limsup {\mu _{k}%
} $, if we take $C=\left\{ \left( 1,0_{n}\right) \right\} ,$ (\ref{3.1})
holds if and only if $\alpha \geq 1,$ we are necessarily in Case 2 when $%
\alpha =1, $ and both cases are possible when $\alpha >1.$

To obtain the lower and upper bound for the RRF, we need the following
definition of epigraphical set.

\begin{definition}[\textbf{Epigraphical set of }$(P)$]
\label{Def_epi}The \emph{epigraphical set} of $(P)$ associated with a
compact base $\mathcal{B}$ of $K^{\ast }$ is the set%
\begin{equation*}
\begin{array}{ll}
E(\overline{A},\overline{b},\mathcal{B}) & :=\left\{ \lambda ^{T}\left[\,
\overline{A}\mid -\overline{b} \, \right] :\lambda \in \mathcal{B}\right\}
+\{0_{n}\}\times \mathbb{R}_{+}.%
\end{array}
\label{3.4}
\end{equation*}
\end{definition}

{The concept of epigraphical set is the adaptation to robust conic LP of the
hypographical set introduced in \cite{CLPT05} to measure the distance to
ill-posedness in the framework of quantitative stability in LSIP. In
contrast with the LP and LSIP adaptations in \cite{GJLP} and \cite{GJLP15},
the epigraphical set $E(\overline{A},\overline{b},\mathcal{B})$\ not only
depends here on the nominal data ($\overline{A}$ and $\overline{b}$), but
also on the chosen compact basis $\mathcal{B}$ of $K^{\ast }.$}

Obviously, $E(\overline{A},\overline{b},\mathcal{B})$\ is a closed convex
set. We observe that $0_{n+1}\notin \mathop{\rm int}E(\overline{A},\overline{%
b},\mathcal{B}).$ To show this, we argue by contradiction and assume that $%
\epsilon \mathbb{B}_{n+1}\subseteq E(\overline{A},\overline{b},\mathcal{B})$
for some $\epsilon >0.$ Then there exist $\widetilde{\lambda }\in \mathcal{B}
$ and $\mu \geq 0$ such that $\left( 0_{n},-\epsilon \right) =\widetilde{%
\lambda }^{T}\left[ \overline{A}\mid -\overline{b}\right] +\left( 0_{n},\mu
\right) $ and we have
\begin{equation*}
\left( 0_{n},1\right) =\frac{1}{\epsilon +\mu }\widetilde{\lambda }^{T}\left[
\overline{A}\mid \overline{b}\right] \in \left\{ \lambda ^{T}\left[
\overline{A}\mid \overline{b}\right] :\lambda \in K^{\ast }\right\} ,
\end{equation*}%
so that $F_{0_{m}}(\overline{A},\overline{b})=\emptyset $\ (contradiction).

\begin{theorem}[\textbf{Lower/Upper bounds for the radius of robust
feasibility}]
\label{th:bound} Let $\mathcal{B}$ be a compact base of $K^{\ast }$. Then,
the RRF satisfies
\begin{equation}
C_{1}\mathop{\rm dist}\left( 0_{n+1},E(\overline{A},\overline{b},\mathcal{B}%
)\right) \leq \rho (\overline{A},\overline{b})\leq C_{2}\mathop{\rm dist}%
\left( 0_{n+1},E(\overline{A},\overline{b},\mathcal{B})\right) ,
\end{equation}%
where
\begin{equation*}
C_{1}=C_{1}(\mathcal{B})=1/\max \left\{ \Vert \sum_{i=1}^{m}\lambda
_{i}u_{i}\Vert :\lambda \in \mathcal{B},\Vert u_{i}\Vert \leq 1\right\}, %
\end{equation*}
\begin{equation*}
C_{2}=C_{2}(\mathcal{B})=1/\min \left\{ \sum_{i=1}^{m}|\lambda_{i}|:\lambda \in \mathcal{B}\right\} .
\end{equation*}
\end{theorem}

{}

\begin{remark}
Before get into the proof of Theorem \ref{th:bound}, we first make the
following remark:

\begin{itemize}
\item (Invariance of the bounds in scaling the compact base) We first
observe that, if $\mathcal{B}$ is a compact base for $K^*$, then $\mu
\mathcal{B}$ is also a compact base for $K^*$ for any $\mu>0$. On the other
hand, note that for any $\mu>0$, $C_i(\mu \mathcal{B})=\frac{1}{\mu} C_i(%
\mathcal{B})$, $i=1,2$ and $\mathop{\rm dist}\left( 0_{n+1},E(\overline{A},%
\overline{b},\mu \mathcal{B})\right)=\mu \mathop{\rm dist}\left( 0_{n+1},E(%
\overline{A},\overline{b},\mathcal{B})\right)$ (see \eqref{3.7}). So, we see
that the lower/upper bounds remain the same if we replace $\mathcal{B}$ by $%
\mu \mathcal{B}$ with any $\mu>0$.

\item (Tightness of the bounds) As we will see in Example 3.1 and Example 3.2, the obtained
bounds can be tight.

\item (Gaps between the lower and upper bounds) Denote the ratio between the
lower bound and upper bound by $\tau$. Then, $\tau=C_1/C_2 \in (0,1]$. In
general, this ratio can depend on the dimension of the cone $K^*$. For
example, if $K$ is the second-order cone in $\mathbb{R}^m$, then, $\tau=%
\frac{1}{\sqrt{m-1}+1}$ (Corollary \ref{Corol_SOCP}). If $K$ is the positive
semi-definite cone $S^q_+$, then, $\tau \in [\frac{2}{q^{3/2}(q+1)},1]$
(Corollary \ref{Corol_SDP}).
\end{itemize}
\end{remark}

\begin{proof}
Let $\rho >\rho (\overline{A},\overline{b}).$ Then, $\rho 1_{m}\notin C(%
\overline{A},\overline{b}).$\ Recall that
\begin{equation*}
C(\overline{A},\overline{b})=\left\{ r\in \mathbb{R}_{+}^{m}:\left(
0_{n},1\right) \notin \overline{\mathop{\rm cone}\bigcup\limits_{\lambda \in
\mathcal{B}}\left\{ \sum\limits_{i=1}^{m}\lambda _{i}\left( (\overline{a}%
_{i},\overline{b}_{i})+r_{i}\mathbb{B}_{n+1}\right) \right\} }\right\} .
\end{equation*}%
This shows that
\begin{equation*}
(0_{n},1)\in \overline{\mathop{\rm cone}\left[ \left\{ \lambda ^{T}\left[
\overline{A}\mid \overline{b}\right] +\rho \sum_{i=1}^{m}\lambda
_{i}u_{i}:\lambda \in \mathcal{B},u_{i}\in \mathbb{B}_{n+1}\right\} \right] }%
,
\end{equation*}
%
%
%
%
%
%
%
%
%
%
%
%
%
%
%
%
%
%
%
%
%
%
%
%
%
%
%
%
%
%
%
%
%
%
%
%
%
%
%
%
%
%
%
%
%
%
%
%
%
%
Thus, letting $w_{1}:=\max \{\Vert \sum_{i=1}^{m}\lambda _{i}u_{i}\Vert
:\lambda \in \mathcal{B},\Vert u_{i}\Vert \leq 1\}$, we have
\begin{equation*}
(0_{n},1)\in \overline{\mathop{\rm cone}\left[ \left\{ \lambda ^{T}\left[
\overline{A}\mid \overline{b}\right] +\rho w_{1}\mathbb{B}_{n+1}:\lambda \in
\mathcal{B}\right\} \right] },
\end{equation*}%
%
%
%
%
%
%
%
%
%
%
%
%
%
%
%
%
%
%
%
%
%
%
%
%
%
%
%
%
%
%
%
%
%
%
%
%
%
%
%
%
%
%
%
%
%
%
%
%
%
%
%
%
%
%
%
%
%
%
%
%
%
%
%
%
%
%
%
where $\left\{ \lambda ^{T}\left[ \overline{A}\mid \overline{b}\right]
:\lambda \in \mathcal{B}\right\} $ is a compact convex set. Thus, according
to Lemma \ref{lemma:1},\
\begin{equation*}
(0_{n},1)\in \mathop{\rm cone}\left[ \left\{ \lambda ^{T}\left[ \overline{A}%
\mid \overline{b}\right] :\lambda \in \mathcal{B}\right\} +(\rho
w_{1}+\epsilon )\mathbb{B}_{n+1}\right] ,\text{ }\forall \epsilon >0,
\end{equation*}%
or, equivalently,
\begin{equation*}
(0_{n},-1)\in \mathop{\rm cone}\left[ \left\{ \lambda ^{T}\left[ \overline{A}%
\mid -\overline{b}\right] :\lambda \in \mathcal{B}\right\} +(\rho
w_{1}+\epsilon )\mathbb{B}_{n+1}\right] ,\text{ }\forall \epsilon >0.
\end{equation*}%
Hence, there exist $\bar{\lambda}^{j}\in \mathcal{B},$ $j=1,...,n,$ $\mu
_{j}\geq 0$, and vectors $\left( u_{j},s_{j}\right) $, $j=1,\ldots ,n+1$,
such that $\Vert \left( u_{j},s_{j}\right) \Vert \leq 1$ and
\begin{equation}
\sum_{j=1}^{n+1}\mu _{j}\sum_{i=1}^{m}\bar{\lambda}_{i}^{j}(\overline{a}%
_{i},-\overline{b}_{i})+(0_{n},1)=-(\rho w_{1}+\epsilon )\sum_{j=1}^{n+1}\mu
_{j}\left( u_{j},s_{j}\right) .
\end{equation}%
Clearly, we observe that $\sum_{j=1}^{n+1}\mu _{j}>0$. Dividing both sides
by $\sum_{j=1}^{n+1}\mu _{j}$, we have
\begin{equation*}
\sum_{i=1}^{m}\gamma _{i}(\overline{a}_{i},-\overline{b}_{i})+\left( 0_{n},{%
\textstyle {\frac{1 }{\sum_{j=1}^{n+1}\mu _{j}}}}\right) =-(\rho
w_{1}+\epsilon )\left( u,s\right) ,
\end{equation*}%
where $(u,s)=\frac{\sum_{j=1}^{n+1}\mu _{j}(u_{j},s_{j})}{%
\sum_{j=1}^{n+1}\mu _{j}}\in \mathbb{B}_{n+1}$ and $\gamma _{i}=\frac{%
\sum_{j=1}^{n+1}\mu _{j}\bar{\lambda}_{i}^{j}}{\sum_{j=1}^{\mu }{}\mu _{j}},$
with $(\gamma _{1},\ldots ,\gamma _{m})\in \mathop{\rm conv}\mathcal{B}=%
\mathcal{B}.$ So, we see that
\begin{equation*}
E(\overline{A},\overline{b},\mathcal{B})\cap (\rho w_{1}+\epsilon )\mathbb{B}%
_{n+1}\neq \emptyset
\end{equation*}%
for all $\epsilon >0,$ which implies that $\mathop{\rm dist}(0_{n+1},E(%
\overline{A},\overline{b},\mathcal{B}))\leq \rho w_{1}+\epsilon$. Letting $%
\epsilon \rightarrow 0$, we see that $\rho \geq \frac{1}{w_{1}}%
\mathop{\rm
dist}(0_{n+1},E(\overline{A},\overline{b},\mathcal{B}))$. Hence,
\begin{equation*}
\rho (\overline{A},\overline{b})\geq \frac{1}{w_{1}}\,\mathop{\rm dist}%
(0_{n+1},E(\overline{A},\overline{b},\mathcal{B})).
\end{equation*}

%
To see the second inequality, let $\rho >0$ be such that $$%
\mathop{\rm
dist}(0_{n+1},E(\overline{A},\overline{b},\mathcal{B}))<\rho.$$ Then, there
exists $(z,s)\in E(\overline{A},\overline{b},\mathcal{B})$ such that $\Vert
(z,s)\Vert \leq \rho$. So, there exist $\widetilde{\lambda }\in \mathcal{B},$
and $\mu \geq 0$ such that $(z,s)=\sum_{i=1}^{m}\widetilde{\lambda }_{i}(%
\overline{a}_{i},-\overline{b}_{i})+(0_{n},\mu )$. Let $\epsilon >0$ be an
arbitrary positive number. Then we have
\begin{equation*}
(0_{n},-(\mu +\epsilon ))\in \sum_{i=1}^{m}\widetilde{\lambda }_{i}(%
\overline{a}_{i},-\overline{b}_{i})+(\rho +\epsilon )\mathbb{B}_{n+1},
\end{equation*}%
Dividing both sides by $\mu +\epsilon $, one has%
\begin{equation*}
\begin{array}{ll}
(0_{n},-1) & \in \left( \frac{1}{\mu +\epsilon }\right) \left( \displaystyle%
\sum_{i=1}^{m}\widetilde{\lambda }_{i}(\overline{a}_{i},-\overline{b}%
_{i})+\left( \rho +\epsilon \right) \mathbb{B}_{n+1}\right) ,%
\end{array}%
\end{equation*}%
Thus, there exists $(u,r)\in \mathbb{B}_{n+1}$ such that
\begin{equation*}
(0_{n},1)\in \left( {\textstyle {\frac{1 }{\mu +\epsilon }}}\right) \left(
\sum_{i=1}^{m}\widetilde{\lambda }_{i}(\overline{a}_{i},-\overline{b}%
_{i})+\left( \rho +\epsilon \right) (u,r)\right) .
\end{equation*}%
Let $w_{2}:=\min \{\sum_{i=1}^{m}|\lambda _{i}|:\lambda \in \mathcal{B}\}>0$
and
\begin{equation*}
(u_{i},r_{i})=\left( {\textstyle {\frac{\mathrm{sign}\widetilde{\lambda }%
_{i} }{\sum_{i=1}^{m}|\widetilde{\lambda }_{i}|}}}u,{\textstyle {\frac{%
\mathrm{sign}\widetilde{\lambda }_{i} }{\sum_{i=1}^{m}|\widetilde{\lambda }%
_{i}|}}}r\right) \in w_{2}^{-1}\mathbb{B}_{n+1}.
\end{equation*}%
Then, we have $\sum_{i=1}^{m}\widetilde{\lambda }_{i}(u_{i},r_{i})=(u,r)$,
and so,
\begin{eqnarray*}
(0_{n},1) &\in &\left( {\textstyle {\frac{1 }{\mu +\epsilon }}}\right)
\left( \sum_{i=1}^{m}\widetilde{\lambda }_{i}(\overline{a}_{i},-\overline{b}%
_{i})+\sum_{i=1}^{m}\widetilde{\lambda }_{i}\left( \rho +\epsilon \right)
(u_{i},r_{i})\right) \\
&\subseteq &\mathop{\rm cone}\left\{ \sum_{i=1}^{m}{\lambda }_{i}\left( (%
\overline{a}_{i},-\overline{b}_{i})+\left( {\textstyle {\frac{\rho +\epsilon
}{w_{2}}}}\right) \mathbb{B}_{n+1}\right) :\lambda \in \mathcal{B}\right\} .
\end{eqnarray*}%
This shows that $\frac{(\rho +\epsilon )}{w_{2}}1_{m}\notin C(\overline{A},%
\overline{b})$, and hence, $\rho (\overline{A},\overline{b})<{\displaystyle {%
\frac{(\rho +\epsilon ) }{w_{2}}}}$. Letting $\epsilon \rightarrow 0$, we
have 
$\rho \geq w_{2}\,\rho (\overline{A},\overline{b})$. 
So, $\mathop{\rm dist}(0_{n+1},E(\overline{A},\overline{b},\mathcal{B}))\geq
w_{2}\,\rho (\overline{A},\overline{b})$, and hence, the conclusion holds.
\end{proof}

\begin{example}
\label{ExampleSOC1}Consider the problem obtained by replacing $\mathbb{R}%
_{+}^{2},$ in Example \ref{example1}, by the second-order cone (SOC), $%
K_{p}^{2}=\left\{ x\in \mathbb{R}^{2}:x_{2}\geq \left\vert x_{1}\right\vert
\right\} .$ We first compute its RRF, say $\rho _{p}(\overline{A},\overline{b%
}),$ by means of Proposition \ref{Prop_Consist_Set3} part (a) .\newline
Let $\alpha >0.$ Taking the compact base $\mathcal{B}=\left[ -1,1\right]
\times \left\{ 1\right\} ,$ one has%
\begin{equation*}
\begin{array}{ll}
\left( \alpha ,\alpha \right) \in C(\overline{A},\overline{b}) &
\Longleftrightarrow \left( 0,1\right) \notin \overline{\mathop{\rm cone}%
\bigcup\limits_{\left\vert \lambda _{1}\right\vert \leq 1}\left\{ \lambda
_{1}\left( (2,0)+\alpha \mathbb{B}_{2}\right) +\left( -1,-3\right) +\alpha
\mathbb{B}_{2}\right\} } \\
& \Longleftrightarrow \left( 0,1\right) \notin \overline{\mathop{\rm cone}%
\left\{ \left( -1,-3\right) +\alpha \mathbb{B}_{2}+\displaystyle%
\bigcup\limits_{\left\vert \lambda _{1}\right\vert \leq 1}\left\{ \lambda
_{1}\left( (2,0)+\alpha \mathbb{B}_{2}\right) \right\} \right\} }. \\
& \Longleftrightarrow \left( 0,1\right) \notin \overline{\mathop{\rm cone}%
\left\{ A+\left( -1,-3\right) +\alpha \mathbb{B}_{2}\right\} },%
\end{array}%
\end{equation*}%
where $A$ $:=\displaystyle\bigcup\limits_{\left\vert \lambda _{1}\right\vert
\leq 1}\left\{ \lambda _{1}\left( (2,0)+\alpha \mathbb{B}_{2}\right)
\right\} $ is the union of $\mathop{\rm conv}\left\{ \left( (2,0)+\alpha
\mathbb{B}_{2}\right) \cup \left\{ 0,0\right\} \right\} $ with its symmetric
set w.r.t. $\left( 0,0\right) .$ $
C(\overline{A},\overline{b}),$ the points of $D$ closest to the line $%
x_{2}=0 $ are $\left( -3,-3+2\alpha \right) $ and $\left( 1,-3+2\alpha
\right) .$ So, $\left( \alpha ,\alpha \right) \in C(\overline{A},\overline{b}%
)$ if and only if $2\alpha \leq 3,$ i.e., $\rho _{p}(\overline{A},\overline{b%
})=\frac{3}{2}.$
\end{example}


Let us compare this exact value of the radius with the result of applying
Theorem \ref{th:bound}. Let $\mathcal{B}$ be the\ "natural" base $\left[ -1,1%
\right] \times \left\{ 1\right\} $ of $K_{p}^{2}$. Then,%
\begin{equation*}
\begin{array}{ll}
E_{p}(\overline{A},\overline{b},\mathcal{B}) & :=\left\{ \lambda ^{T}\left[
\overline{A}\mid -\overline{b}\right] :\lambda \in \mathcal{B}\right\}
+\{0\}\times \mathbb{R}_{+} \\
& =\mathop{\rm conv}\left\{ \lambda ^{T}\left[ \overline{A}\mid -\overline{b}%
\right] :\lambda =\left( \pm 1,1\right) \right\} +\{0\}\times \mathbb{R}_{+}
\\
& =\left( \left[ -3,1\right] \times 3\right) +\{0\}\times \mathbb{R}_{+} \\
& =\left[ -3,1\right] \times \left[ 3,+\infty \right) ,%
\end{array}%
\end{equation*}%
with $\mathop{\rm dist}\left( 0_{n+1},E(\overline{A},\overline{b},\mathcal{B}%
)\right) =3$. In this case, direct calculation shows that
\begin{equation*}
C_{1}=1/\left\{ \max \Vert \lambda _{1}u_{1}+u_{2}\Vert :\lambda _{1}\in
\lbrack -1,1],\Vert u_{1}\Vert \leq 1,\Vert u_{2}\Vert \leq 1\right\} =\frac{%
1}{2}
\end{equation*}%
and $C_{2}=1/\min \{|\lambda _{1}|+1:\lambda _{1}\in \lbrack -1,1]\}=1$.
Thus, our previous result shows that
\begin{equation*}
\rho _{p}(\overline{A},\overline{b})\in \left[ \frac{3}{2},3\right] ,
\end{equation*}%
i.e., the lower bound \thinspace $C_{1}\mathop{\rm dist}\left( 0_{n+1},E(%
\overline{A},\overline{b},\mathcal{B})\right) $\ is exact here.

The lower and the upper bounds for the RRF provided by Theorem \ref{th:bound}
involve two constants depending on the chosen base $\mathcal{B}$ of $K,$ $%
C_{1}$ and $C_{2},$ and a non-negative number, $\mathop{\rm dist}\left(
0_{n+1},E(\overline{A},\overline{b},\mathcal{B})\right) ,$ which also
depends on the nominal matrix $\left[ \overline{A},\overline{b}\right] .$ We
now provide a computable formula for the distance from the epigraphical set
to the origin.

\begin{theorem}[\textbf{A computable formula for} $\mathop{\rm dist}\left(
0_{n+1},E(\overline{A},\overline{b},\mathcal{B})\right) $]
\label{TheorComput} Let $\mathcal{B}$\ be a compact base of $K^{\ast }.$
Then,
{\small \begin{equation}
\mathop{\rm dist}\left( 0_{n+1},E(\overline{A},\overline{b},\mathcal{B}%
)\right) =\inf\limits_{(z,s,t,\lambda )\in \mathbb{R}^{n}\times \mathbb{R}%
\times \mathbb{R}\times \mathbb{R}^{m}}\left\{ t\left\vert
\begin{array}{l}
\Vert (z,s)\Vert \leq t,\medskip \\
z=\overline{A}^T\lambda,\ s\geq -\overline{b}^{T}\lambda,, \\
\lambda \in \mathcal{B}.%
\end{array}%
\right. \right\} .  \label{3.7}
\end{equation}}%
In particular, let $\mathcal{B}$ is a spectrahedron with the form $$%
\mathcal{B} =\{\lambda \in \mathbb{R}^{m}:B_{0}+\sum_{i=1}^{m}\lambda
_{i}B_{i}\succeq 0\}$$ for some $(s\times s)$ symmetric matrices $B_{i}$, $%
i=0,1,\ldots ,m$. Then one has $$\mathop{\rm dist}\left( 0_{n+1},E(\overline{A},%
\overline{b},\mathcal{B})\right) =\sqrt{f^{\ast }},$$ where $f^{\ast }$ is the
optimal value of the following semi-definite program:
\begin{equation}
\inf\limits_{(z,s,t,\lambda )\in \mathbb{R}^{n}\times \mathbb{R}\times
\mathbb{R}\times \mathbb{R}^{m}}\left\{ t\left\vert
\begin{array}{l}
\left[
\begin{array}{ccc}
tI_{n} & 0_{n} & z \\
0_{n}^{T} & t & s \\
z^{T} & s & 1%
\end{array}%
\right] \succeq 0\medskip \\
z=\overline{A}^T\lambda,\ s\geq -\overline{b}^{T}\lambda, \\
\displaystyle B_{0}+\sum_{i=1}^{m}\lambda _{i}B_{i}\succeq 0.%
\end{array}%
\right. \right\} .  \label{3.70}
\end{equation}
\end{theorem}

\begin{proof}
By definition,
\begin{eqnarray}
\ \ \ \ \ \ \ \mathop{\rm dist}\left( 0_{n+1},E(\overline{A},\overline{b},%
\mathcal{B})\right) &=&\inf\limits_{(z,s)\in \mathbb{R}^{n}\times \mathbb{R}%
}\{\Vert (z,s)\Vert :(z,s)\in E(\overline{A},\overline{b},\mathcal{B}%
)\}\medskip \\
&=&\inf\limits_{(z,s,t)\in \mathbb{R}^{n}\times \mathbb{R}\times \mathbb{R}%
}\{t:t\geq \Vert (z,s)\Vert ,(z,s)\in E(\overline{A},\overline{b},\mathcal{B}%
)\}.  \notag
\end{eqnarray}%
From the definition of $E(\overline{A},\overline{b},\mathcal{B})$, one has
\begin{equation}
(z,s)\in E(\overline{A},\overline{b},\mathcal{B})\Leftrightarrow \exists
\,w\geq 0,\,\lambda \in \mathcal{B}:\text{ }(z,s)=\sum_{i=1}^{m}\lambda _{i}(%
\overline{a}_{i},-\overline{b}_{i})+(0_{n},w).  \label{3.9}
\end{equation}%
So, \eqref{3.7} holds.

To see the second assertion, we note that
\begin{eqnarray}
& & \mathop{\rm dist}\left( 0_{n+1},E(\overline{A},\overline{b},\mathcal{B}%
)\right) ^{2} \nonumber \\ &=&\inf\limits_{(z,s)\in \mathbb{R}^{n}\times \mathbb{R}%
}\{\Vert (z,s)\Vert ^{2}:(z,s)\in E(\overline{A},\overline{b},\mathcal{B}%
)\}\medskip  \notag \\
&=&\inf\limits_{(z,s,t)\in \mathbb{R}^{n}\times \mathbb{R}\times \mathbb{R}%
}\{t:t\geq \Vert (z,s)\Vert ^{2},(z,s)\in E(\overline{A},\overline{b},%
\mathcal{B})\},  \label{3.8}
\end{eqnarray}%
where $t\geq \Vert (z,s)\Vert ^{2}$ can be replaced by
\begin{equation*}
\left[
\begin{array}{ccc}
tI_{n} & 0_{n} & z \\
0_{n}^{T} & t & s \\
z^{T} & s & 1%
\end{array}%
\right] \succeq 0
\end{equation*}%
thanks to the Schur complement (see, e.g., \cite[Lemma 4.2.1]{BN01}). So,
one gets the desired conclusion $\mathop{\rm dist}\left( 0_{n+1},E(\overline{%
A},\overline{b},\mathcal{B})\right) =\sqrt{f^{\ast }}$ from (\ref{3.8}) and (%
\ref{3.9}).
\end{proof}

As a corollary, we obtain the radius formula of robust feasibility for an
uncertain linear program which was established in \cite{GJLP}. Here, we
further show that this formula leads to efficient computation of tight upper
bounds for the RRF via second order cone programs.

\begin{corollary}[\textbf{Computable exact formula for} $\protect\rho (%
\overline{A},\overline{b})$ \textbf{of uncertain LPs}]
\label{cor:18} Let $K=\mathbb{R}_{+}^{m}$. Then, the RRF satisfies
\begin{equation*}
\rho (\overline{A},\overline{b})=\mathop{\rm dist}(0_{n+1},E(\overline{A},%
\overline{b},\mathcal{B}))=f_{LP}^{\ast },
\end{equation*}%
where $f_{LP}^{\ast }$ is the optimal value of the following second order
programming problem:
\begin{equation}
\inf\limits_{(z,s,t,\lambda )\in \mathbb{R}^{n}\times \mathbb{R}\times
\mathbb{R}\times \mathbb{R}^{m}}\left\{ t\ \left\vert
\begin{array}{l}
\Vert (z,s)\Vert \leq t,\medskip \\
z=\overline{A}^{T}\lambda ,\ s\geq -\overline{b}^{T}\lambda , \\
\lambda \in \mathbb{R}_{+}^{m},\ \sum_{i=1}^{m}\lambda _{i}=1%
\end{array}%
\right. \right\} .  \label{eq:996}
\end{equation}
\end{corollary}

\begin{proof}
Let $K=\mathbb{R}_{+}^{m}$. Then, $K^{\ast }=K=\mathbb{R}_{+}^{m}$. Then,
the simplex $\Delta =\{\lambda :\lambda \in \mathbb{R}_{+}^{m},%
\sum_{i=1}^{m}\lambda _{i}=1\}$ is a natural compact base for $\mathbb{R}%
_{+}^{m}$. 
Thus, the conclusion follows by letting $\mathcal{B}=\Delta$ in \eqref{3.7}
in Theorem 
\ref{TheorComput}.
\end{proof}

We now provide a simple example illustrating Corollary \ref{cor:18}.

\begin{example}
Consider the same problem examined in Example \ref{example1}. Then, the
second order problem in (\ref{eq:996}) becomes here
\begin{equation*}
\inf_{(z,s,t,\lambda )\in \mathbb{R}\times \mathbb{R}\times \mathbb{R}\times
\mathbb{R}^{2}}\left\{ t\left\vert
\begin{array}{l}
\Vert (z,s)\Vert \leq t, \\
z=2\lambda _{1}-\lambda _{2},\ s\geq 3\lambda _{2}, \\
\lambda _{1}+\lambda _{2}=1,\lambda _{1}\geq 0,\lambda _{2}\geq 0,%
\end{array}%
\right. \right\} ,
\end{equation*}%
whose optimal set is $\left\{ 1\right\} \times \left\{ 1\right\} \times
\left\{ \sqrt{2}\right\} \times \left\{ \left( \frac{2}{3},\frac{1}{3}%
\right) \right\} .$ Thus, we get again $\rho (\overline{A},\overline{b})=%
\sqrt{2}$ which coincides with the computation in Example \ref{Example1}.
\end{example}

\noindent

\section{Bounds for Radius of Robust Feasibility of SDPs \& SOCPs}

We now consider uncertain linear semi-definite programming problems and
provide computable bounds for the RRF. To do this, recall that $S_{+}^{q}$
is the cone which consists of all $(q\times q)$ positive semi-definite
matrices. Denote the set of all $(q\times q)$ symmetric matrices by $S^{q}$
and let $\mathrm{Tr}(M)$ be the trace of a matrix $M\in S^{q}$. As $S^{q}$
and $\mathbb{R}^{q(q+1)/2}$ have the same dimensions, there exists an
invertible linear map $L:S^{q}\rightarrow \mathbb{R}^{q(q+1)/2}$ such that
\begin{equation*}
L(M_{1})^{T}L(M_{2})=\mathrm{Tr}(M_{1}M_{2})\mbox{ for all }M_{1},M_{2}\in
S^{q}.
\end{equation*}%
We now identify the space of $(q\times q)$ symmetric matrices $S^{q},$
equipped with the trace inner product, as $\mathbb{R}^{q(q+1)/2}$ with the
usual Euclidean inner product by associating each symmetric matrix $M$ to $%
L(M)$.

\begin{corollary}[\textbf{Numerically tractable bounds for} $\protect\rho (%
\overline{A},\overline{b})$ \textbf{of uncertain SDPs}]
\label{Corol_SDP}Identify $S^{q}$ with $\mathbb{R}^{q(q+1)/2}$ via the
mapping $L$ given as above and let $K$ be the positive semi-definite cone $%
S_{+}^{q}$. Let $\mathcal{B}=\{\Lambda \in S^{q}:\Lambda \in S_{+}^{q},%
\mathrm{Tr}(\Lambda )=1\}$ be the natural compact base for $K$. Then, the
RRF satisfies
\begin{equation*}
\frac{2}{q(q+1)}\sqrt{f_{SDP}^{\ast }}\leq \rho (\overline{A},\overline{b}%
)\leq \sqrt{q}\,\sqrt{f_{SDP}^{\ast }} \ ,
\end{equation*}%
where $f_{SDP}^{\ast }$ is the optimal value of the following semi-definite
program:
\begin{equation*}
\inf\limits_{\substack{ (z,s,t,\lambda )\in \mathbb{R}^{n}\times \mathbb{R}%
\times \mathbb{R}\times \mathbb{R}^{m}  \\ \Lambda \in S^{q}}}\left\{
t\left\vert
\begin{array}{l}
\left[
\begin{array}{ccc}
tI_{n} & 0_{n} & z \\
0_{n}^{T} & t & s \\
z^{T} & s & 1%
\end{array}%
\right] \succeq 0,\medskip \\
z=\overline{A}^{T}\lambda ,\ s\geq -\overline{b}^{T}\lambda , \\
\lambda =L(\Lambda ) \\
\Lambda \in S_{+}^{q},\ \mathrm{Tr}(\Lambda )=1.%
\end{array}%
\right. \right\} .
\end{equation*}
\end{corollary}

\begin{proof}
Let $K$ be the positive semi-definite cone $S_{+}^{q}$. Then, $K^{\ast
}=K=S_{+}^{q}$. Let $\mathcal{B}=\{\Lambda \in S^{q}:\Lambda \in S_{+}^{q},%
\mathrm{Tr}(\Lambda )=1\}$ be the natural compact base for $S_{q}^{+}$.
Then, Theorem \ref{th:bound} and Theorem \ref{TheorComput} imply that
\begin{equation*}
C_{1}\sqrt{f_{SDP}^{\ast }}\leq \rho (\overline{A},\overline{b})\leq C_{2}%
\sqrt{f_{SDP}^{\ast }}\, ,
\end{equation*}%
where $m=\frac{q(q+1)}{2}$,
\begin{equation*}
C_{1}=1/\max \left\{ \Vert \sum_{i=1}^{m}\lambda _{i}u_{i}\Vert :\lambda
=L(\Lambda ),\,\Lambda \in \mathcal{B},\Vert u_{i}\Vert \leq 1\right\}
\label{eq:000}
\end{equation*}%
and
\begin{equation*}
C_{2}=1/\min \left\{ \sum_{i=1}^{m}|\lambda _{i}|:\lambda =L(\Lambda
),\,\Lambda \in \mathcal{B}\right\} .
\end{equation*}%
To see the conclusion, it suffices to show that $C_{1}\geq \frac{2}{q(q+1)}$
and $C_{2}\leq \sqrt{q}$. To see this, from the definition of $L$, we have $%
\Vert \lambda \Vert ^{2}=\Vert L(\Lambda )\Vert ^{2}=\mathrm{Tr}(\Lambda
^{2})$. Let $\Lambda =U\Sigma U^{T}$ be the singular value decomposition of $%
\Lambda \in S^{q}$ where $U$ is an orthonormal matrix and $\Sigma $ is a
diagonal matrix whose diagonal elements are the eigenvalues of $\Lambda $.
Then, $\Lambda ^{2}=U\Sigma ^{2}U^{T}$. It follows from the
arithmetic-quadratic inequality that for all $\Lambda \in \mathcal{B}$,
\begin{equation*}
\frac{1}{q}=\frac{1}{q}\mathrm{Tr}(\Lambda )^{2}\leq \mathrm{Tr}(\Lambda
^{2})\leq \lbrack \mathrm{Tr}(\Lambda )]^{2}=1.
\end{equation*}%
This implies that $\frac{1}{\sqrt{q}}\leq \Vert \lambda \Vert \leq 1$ for
all $\Lambda \in \mathcal{B}$. Thus, one has
\begin{equation*}
C_{1}\geq 1/\max \left\{ \sum_{i=1}^{m}|\lambda _{i}|:\lambda =L(\Lambda
),\,\Lambda \in \mathcal{B}\right\} \geq \frac{1}{m}=\frac{2}{q(q+1)}.
\end{equation*}%
%
%
%
%
%
%
%
%
%
%
%
%
%
%
%
%
%
%
%
%
%
%
%
%
%
%
%
%
%
%
%
%
%
%
%
%
%
%
%
Moreover, as $\sum_{i=1}^{m}|\lambda _{i}|\geq \Vert \lambda \Vert \geq
\frac{1}{\sqrt{q}}$, one has
\begin{equation*}
C_{2}=1/\min \left\{ \sum_{i=1}^{m}|\lambda _{i}|:\lambda =L(\Lambda
),\,\Lambda \in \mathcal{B}\right\} \leq \sqrt{q}.
\end{equation*}%
So, the conclusion follows.
\end{proof}


\begin{example}
Consider the following semi-definite program (SDP)
\begin{equation*}
(SDP)\ \ \ \min \left\{ -x_{1}:\left[
\begin{array}{cc}
-1 & x_{2} \\
x_{2} & -1+x_{1}%
\end{array}%
\right] \in -S_{+}^{2}\right\} .
\end{equation*}%
We now investigate the RRF for this SDP.

Let $L:S^{2}\rightarrow \mathbb{R}^{3}$ be defined as $L\left( \left[
\begin{array}{cc}
z_{1} & z_{2} \\
z_{2} & z_{3}%
\end{array}%
\right] \right) =(z_{1},\sqrt{2}z_{2},z_{3})^{T}$. Clearly, $L$ is a
one-to-one mapping with $L(M_{1})^{T}L(M_{2})=\mathrm{Tr}(M_{1}M_{2})$. So,
the feasible set of $(SDP)$ can be written as $\{x\in \mathbb{R}^{2}:%
\overline{A}x+\overline{b}\in -L\left( S_{+}^{2}\right) \},$ with $\overline{%
A}=\left[
\begin{array}{cc}
0 & 0 \\
0 & \sqrt{2} \\
1 & 0%
\end{array}%
\right] ,$ $\overline{b}=\left(
\begin{array}{c}
-1 \\
0 \\
-1%
\end{array}%
\right) $ and $L\left( S_{+}^{2}\right) =\left\{ y\in \mathbb{R}%
^{3}:y_{2}^{2}\leq 2y_{1}y_{3},y_{1}\geq 0,y_{3}\geq 0\right\} .$ Direct
computation shows that $f_{\mathrm{SDP}}^{\ast }=1$ in this case with an
optimal solution $(z,s,t,\lambda )=(0_{2},1,1,(1,0,0)^{T})$. Then, the
preceding corollary implies that the RRF satisfies $\frac{1}{3}\leq \rho (%
\overline{A},\overline{b})\leq \sqrt{2}$.

On the other hand, it can be directly verified that the true RRF for this
example satisfies $\rho (\overline{A},\overline{b})\in \lbrack \frac{1}{2}%
,1] $. Indeed, let $\Delta b(\epsilon )=(1+\epsilon ,0,0)^{T}$ for any $%
\epsilon >0$. note that $\{x:\overline{A}x+\overline{b}+\Delta b(\epsilon
)\in -S_{+}^{2}\}=\emptyset $. This shows that $\rho (\overline{A},\overline{%
b})\leq 1+\epsilon $ for all $\epsilon >0$, and hence $\rho (\overline{A},%
\overline{b})\leq 1$. Moreover, let $\overline{A}=\left[ \overline{a}%
_{1}\mid \overline{a}_{2}\mid \overline{a}_{3}\right] ^{T}$ and $A=\left[ {a}%
_{1}\mid {a}_{2}\mid {a}_{3}\right] ^{T}$. Then, for all $(a_{i},b_{i})\in (%
\overline{a}_{i},\overline{b}_{i})+\frac{1}{2}\mathbb{B}_{3}$, one sees that
$b_{1}\in \lbrack -\frac{3}{2},-\frac{1}{2}]$, $b_{2}\in \lbrack -\frac{1}{2}%
,\frac{1}{2}]$ and $b_{2}\in \lbrack -\frac{3}{2},-\frac{1}{2}]$. So, $%
0_{2}\in \{x:Ax+b\in -L\left( S_{+}^{2}\right) \}$. This shows that $\rho (%
\overline{A},\overline{b})\geq \frac{1}{2}$.
\end{example}

Next, we see that, in the case for an uncertain second-order cone program (that
is, $K=K_{p}^{m}:=\left\{ x\in \mathbb{R}^{m}:x_{m}\geq \left\Vert \left(
x_{1},...,x_{m-1}\right) \right\Vert \right\} $), the lower and upper bound
of the RRF can be computed by solving a second-order cone programming
problem.

\begin{corollary}[\textbf{Numerically tractable bounds for }$\protect\rho (%
\overline{A},\overline{b})$ \textbf{of uncertain SOCPs}]
\label{Corol_SOCP}Let $K$ be the second-order cone $K_{p}^{m}$ in $\mathbb{R}%
^{m}$, and let $\mathcal{B}=\{\lambda \in \mathbb{R}^{m}:\Vert (\lambda
_{1},\ldots ,\lambda _{m-1})\Vert \leq 1\mbox{ and }\lambda _{m}=1\}$ be the
natural compact base for $K_{p}^{m}$. Then, the RRF satisfies
\begin{equation*}
\frac{1}{\sqrt{m-1}+1}f_{SOC}^{\ast }\leq \rho (\overline{A},\overline{b}%
)\leq f_{SOC}^{\ast },
\end{equation*}%
where $f_{SOC}^{\ast }$ is the optimal value of the following second order
cone program:
\begin{equation*}
\inf\limits_{(z,s,t,\lambda )\in \mathbb{R}^{n}\times \mathbb{R}\times
\mathbb{R}\times \mathbb{R}^{m}}\left\{ t\left\vert
\begin{array}{l}
\Vert (z,s)\Vert \leq t,\medskip \\
z=\overline{A}^{T}\lambda ,\ s\geq -\overline{b}^{T}\lambda , \\
\Vert (\lambda _{1},\ldots ,\lambda _{m-1})\Vert \leq 1,\ \lambda _{m}=1.%
\end{array}%
\right. \right\} .
\end{equation*}
\end{corollary}

\begin{proof}
Let $K$ be the second-order cone $K_{p}^{m}$ in $\mathbb{R}^{m}$. Then, $%
K^{\ast }=K=K_{p}^{m}$. Let $\mathcal{B}=\{\lambda \in \mathbb{R}^{m}:\Vert
(\lambda _{1},\ldots ,\lambda _{m-1})\Vert \leq 1\mbox{ and }\lambda
_{m}=1\} $ be the natural compact base for $K_{p}^{m}$. We claim that
\begin{equation}
C_{1}=1/\max \left\{ \Vert \sum_{i=1}^{m}\lambda _{i}u_{i}\Vert :\lambda \in
\mathcal{B},\Vert u_{i}\Vert \leq 1\right\} =\frac{1}{\sqrt{m-1}+1}.
\label{eq:00}
\end{equation}%
Indeed, by the triangle inequality, one has
\begin{eqnarray*}
& & \max \left\{ \Vert \sum_{i=1}^{m}\lambda _{i}u_{i}\Vert :\lambda \in
\mathcal{B},\Vert u_{i}\Vert \leq 1\right\} \\
&=&\max \left\{ \Vert
\sum_{i=1}^{m-1}\lambda _{i}u_{i}+u_{m}\Vert :\Vert (\lambda _{1},\ldots
,\lambda _{m-1})\Vert \leq 1,\Vert u_{i}\Vert \leq 1\right\} \\
&\leq &\max \left\{ \sum_{i=1}^{m-1}|\lambda _{i}|+1:\Vert (\lambda
_{1},\ldots ,\lambda _{m-1})\Vert \leq 1\right\} \\
&=&\sqrt{m-1}+1.
\end{eqnarray*}%
Moreover, for $u_{i}=u$ with $\Vert u\Vert =1$ and $\lambda =(%
\underbrace{\frac{1}{\sqrt{m-1}},\ldots ,\frac{1}{\sqrt{m-1}}}_{m-1},1)\in
\mathcal{B}$,
\begin{equation*}
\Vert \sum_{i=1}^{m}\lambda _{i}u_{i}\Vert =(\sqrt{m-1}+1)\Vert u\Vert =%
\sqrt{m-1}+1.
\end{equation*}%
Thus, \eqref{eq:00} holds. Direct verification also shows that $$%
C_{2}=1/\min \left\{ \displaystyle \sum_{i=1}^{m}|\lambda _{i}|:\lambda \in
\mathcal{B}\right\} =1.$$ Therefore, the conclusion follows from Theorem \ref%
{th:bound} and Theorem \ref{TheorComput} (equation \eqref{3.7}).
%
\end{proof}

%
%

{}

\subsection*{\bf Robust Separability in Uncertain SVMs}

The support vector machine for binary classification is a useful technique
in generating an optimal classifier (hyperplane) which separates the
training data into two classes. It has found numerous applications in
engineering, medical imaging and computer science. Let $(u_{i},\alpha
_{i})\in \mathbb{R}^{s}\times \{-1,1\}$ be the given training data where $%
\alpha _{i}$ is the class label for each data $u_{i}$. An optimization model
problem of support vector machine for binary classification can be stated as
follows \cite[Section 12.1.1]{BEL09}:
\begin{eqnarray*}
&\min_{(w,\gamma )\in \mathbb{R}^{s}\times \mathbb{R}}&\ \Vert w\Vert \\
&\mbox{ s.t. }&\alpha _{i}(u_{i}^{T}w+\gamma )\geq 1,i=1,\ldots ,m.
\end{eqnarray*}%
In practice, the given data $u_{i}$, $i=1,\ldots ,m$, are often uncertain.
We assume that these data are subject to the following norm data
uncertainty:
\begin{equation*}
u_{i}\in \mathcal{V}_{i}(r)=\overline{u}_{i}+r_{i}\mathbb{B}_{s}.
\end{equation*}%
%
%
%
%
%
%
%
%
%
%
%
%
%
%
Let $r=(r_{1},\ldots ,r_{m}) \in \mathbb{R}^m_+.$ Then, the robust support vector machine can
be stated as
\begin{eqnarray*}
(RSVM_{r}) &\displaystyle \min_{(w,\gamma )\in \mathbb{R}^{s}\times \mathbb{R}}&\ \Vert
w\Vert \\
&\mbox{ s.t. }&\alpha _{i}(u_{i}^{T}w+\gamma )\geq 1,\ \forall \,u_{i}\in
\mathcal{V}_{i}(r),\ i=1,\ldots ,m.
\end{eqnarray*}%
If the feasible set $F_{r}$ of $(RSVM_{r})$ is nonempty, then a linear
binary classification is possible even if the training data is subject to
measurement or prediction error and the error level is controlled by $r$.
Thus, the range of values of $r$, guaranteeing the non-emptiness of the
feasible set $F_{r}$ of $(RSVM_{r})$, quantifies the robustness of the
linear separability of the training data in uncertain support vector machine
problems. So, we now investigate the question: for what values of $r > 0$  so
that the feasible set $F_{r}$ of $(RSVM_{r})$ is nonempty.

The robust SVM problem can be further rewritten into a robust conic
programming problem as follows
\begin{eqnarray*}
(RSVM_{r}) &\displaystyle \min_{(w,\gamma ,t)\in \mathbb{R}^{s}\times \mathbb{R}\times
\mathbb{R}}&\ t \\
&\mbox{ s.t. }&\Vert w\Vert \leq t \\
&&(-\alpha _{i}u_{i})^{T}w+(-\alpha _{i})\gamma +1\leq 0,\ \forall
\,u_{i}\in \mathcal{V}_{i}(r),\ i=1,\ldots ,m.
\end{eqnarray*}%
Let $K=K_{p}^{s+1}\times \mathbb{R}_{+}^{m}$, where $K_{p}^{s+1}$ is the
second-order cone in $\mathbb{R}^{s+1}$. Let $\overline{a}%
_{i}=(-e_{i}^{T},0,0)^{T}$ and $\overline{b}_{i}=0$ for $i=1,\ldots ,s$; $%
\overline{a}_{s+1}=(0_{s}^{T},0,-1)^{T}$ and $\overline{b}_{s+1}=0$; $%
\overline{a}_{i}=(-\alpha _{i-s-1}u_{i-s-1}^{T},-\alpha _{i-s-1},0)^{T}$ and
$\overline{b}_{i}=1$ for $i=s+2,\ldots ,m+s+1$. Define
\begin{equation*}
\mathcal{U}_{i}(\mu_{i})=(\overline{a}_{i},\overline{b}_{i})+\mu_{i}\mathbb{B%
}_{s+3}, \ i=1,\ldots,m+s+1,
\end{equation*}
where $\mu_i \ge 0$. We now consider a closely related robust SOCP problem
\begin{eqnarray*}
(SOCP_{\mu}) &\min_{(w,\gamma ,t)\in \mathbb{R}^{s}\times \mathbb{R}\times
\mathbb{R}}&t \\
&\mbox{ s.t. }&\left[
\begin{array}{c}
a_{1}^{T}x+b_{1} \\
\vdots \\
a_{m+s+1}^{T}x+b_{m+s+1}%
\end{array}%
\right] \in -K,\forall \ (a_{i},b_{i})\in \mathcal{U}_{i}(\mu_{i}).
\end{eqnarray*}%
Denote the feasible set of $(SOCP_{\mu})$ as $F_{\mu}^{\prime }$ with $%
\mu=(\mu_{1},\ldots ,\mu_{m+s+1})$. Then,  for any $%
r=(r_1,\ldots,r_m) \in \mathbb{R}^{m}$ and $\overline{r}=(\underbrace{%
0,\ldots,0}_{s+1},r_1,\ldots,r_m) \in \mathbb{R}^{m+s+1}$ \vspace{-0.2cm}
\begin{equation}
F_{\overline{r}}^{\prime }\neq \emptyset \Longrightarrow F_{r}\neq \emptyset.
\label{u1}
\end{equation}
\vspace{-0.1cm}Define $\overline{A}=\left[ \overline{a}_{1}\mid ...\mid
\overline{a}_{m+s+1}\right] ^{T}$, $\overline{b}=\left( \overline{b}_{1},...,%
\overline{b}_{m+s+1}\right) ^{T}$ and let $\rho (\overline{A},\overline{b})$
be the robust feasibility of $(SOCP_{\mu})$. Then for all $r\in \big[0,\rho (%
\overline{A},\overline{b})\big)^{m}$, the feasible set $F_{r}$ of the robust
support vector machine problem $(RSVM_{r})$ will be nonempty. In particular,
we have the following result:

\begin{corollary}
Let $r=(r_1,\ldots,r_{m})$ and denote the feasible set of the robust support
vector machine problem $(RSVM_r)$ by $F_r$. Then, $F_r \neq \emptyset$ for
all $r \in \mathbb{R}^m_+$ with $r_i \le \frac{1}{\sqrt{s}+1} \, f^*,$ where
$f^*$ is the optimal value of the following second-order cone program
\begin{equation*}
\inf\limits_{(z,s,t,\lambda )\in \mathbb{R}^{n}\times \mathbb{R}\times
\mathbb{R}\times \mathbb{R}^{m+s+1}}\left\{ t\left\vert
\begin{array}{l}
\Vert (z,s)\Vert \leq t,\medskip \\
z=\overline{A}^T\lambda,\ s\geq -\overline{b}^{T}\lambda, \\
\|(\lambda_1,\ldots,\lambda_s)\| \le \lambda_{s+1}, \\
\sum_{i=1}^{m+1}\lambda_{s+i}=1, \lambda_{s+1}\ge 0 ,\ldots,\lambda_{m+s+1}
\ge 0.%
\end{array}%
\right. \right\} .
\end{equation*}
\end{corollary}

\begin{proof}
Note that $K^{\ast }=K=K_{p}^{s+1}\times \mathbb{R}_{+}^{m}$. So, a compact
base $\mathcal{B}$ for $K^{\ast }$ is
\begin{equation*}
\{\lambda :\Vert (\lambda _{1},\ldots ,\lambda _{s})\Vert \leq \lambda
_{s+1},\sum_{i=1}^{m+1}\lambda _{s+i}=1,\lambda _{s+1}\geq 0,\ldots ,\lambda
_{m+s+1}\geq 0\}.
\end{equation*}%
Let $\rho (\overline{A},\overline{b})$ be the RRF of $(SOCP_{\mu })$. From
Theorem \ref{th:bound}, one sees that
\begin{equation*}
\rho (\overline{A},\overline{b})\geq C_{1}\,\mathop{\rm dist}\left(
0_{n+1},E(\overline{A},\overline{b},\mathcal{B})\right) =C_{1}f^{\ast },
\end{equation*}
where the equality follows from Theorem \ref{TheorComput}. So, for all $%
r_{i} $ with $0\leq r_{i}\leq C_{1}\,f^{\ast }$, $i=1,\ldots ,m$, $F_{%
\overline{r}}^{\prime }\neq \emptyset $ where $\overline{r}=(0,\ldots
,0,r_{1},\ldots ,r_{m})\in \mathbb{R}^{m+s+1}$. This together with \eqref{u1}
implies that $F_{r}$ is nonempty if $0\leq r_{i}\leq C_{1}\,f^{\ast }$, $%
i=1,\ldots ,m$. Now the conclusion follows by noting that
\begin{equation*}
C_{1}=1/\max \left\{ \Vert \sum_{i=1}^{m+s+1}\lambda _{i}u_{i}\Vert :\lambda
\in \mathcal{B},\Vert u_{i}\Vert \leq 1\right\} =\frac{1}{\sqrt{s}+1},
\end{equation*}%
where the last equality holds by using the same reasoning as in the proof of
Corollary \ref{Corol_SOCP}.
\end{proof}

\section{Conclusions}
In this paper, we introduced the notion of radius of robust feasibility for an
uncertain linear conic program, which provides a numerical value for the
largest size of a ball uncertainty set that guarantees non-emptiness of the robust feasible set.
We then provided formulas for estimating the radius of robust feasibility of uncertain conic programs, using the tools of convex analysis and parametric optimization. We also established
computationally tractable bounds for two important uncertain conic programs:
semi-definite programs and second-order cone programs.
In the special case of uncertain linear programs, the formula allows
us to calculate the radius by finding the optimal value of an associated
second-order cone program.

Our results suggest some interesting further work. For example,
the radius of robust feasibility formula was achieved for commonly used ball
uncertainty set. It would be of interest to examine how our approach can be extended
to cover other commonly used uncertainty sets such as the polytope uncertainty
sets or intersection of polytope and norm uncertainty sets. Another interesting topic of study would be to extend our approach to calculate the
radius of robust feasibility for uncertain discrete nonlinear optimization
problems (see a promising computational approach initialized in \cite{LST20} for the linear cases).

\bigskip

\appendix
\noindent Appendix: Proofs of Basic Properties of Admissible Sets

\noindent In this appendix, we provide the proof of the basic properties of the
admissible set of parameters (Proposition \ref{Prop_Consist_Set3}).

\newpage
\medskip \noindent \textbf{Proof of Proposition \ref{Prop_Consist_Set3}:}

\begin{proof}
\ [Proof of \textrm{(a)}] Direct verification gives us that
\begin{eqnarray*}
F_{r}(\overline{A},\overline{b})=\{ x\in \mathbb{R}^{n} &:& \displaystyle%
\sum\limits_{i=1}^{m}\lambda _{i}(\overline{a}_{i}+\Delta a_{i})^{T}x+%
\displaystyle\sum\limits_{i=1}^{m}\lambda _{i}(\overline{b}_{i}+\Delta
b_{i})\leq 0, \\
& &  \forall \lambda \in K^{\ast },\left\Vert (\Delta
a_{i},\Delta b_{i})\right\Vert \leq r_{i},\ i=1,\ldots,m\} .  \label{2.2}
\end{eqnarray*}%
It now follows from the well-known existence theorem for linear systems \cite%
[Corollary 3.1.1]{GL98} that $F_{r}(\overline{A},\overline{b})\neq \emptyset
$ if and only if%
\begin{equation*}
\left( 0_{n},-1\right) \notin \overline{\mathop{\rm cone}\left\{ %
\displaystyle\sum\limits_{i=1}^{m}\lambda _{i}\left( \overline{a}_{i}+\Delta
a_{i},-\overline{b}_{i}-\Delta b_{i}\right) :\lambda \in \mathcal{B},\Vert
(\Delta a_{i},\Delta b_{i})\Vert \leq r_{i},\forall i\right\} },
\end{equation*}%
which, in turn, is equivalent to the statement
\begin{equation*}
\begin{array}{ll}
\left( 0_{n},1\right) & \notin \overline{\mathop{\rm cone}\left\{ %
\displaystyle\sum\limits_{i=1}^{m}\lambda _{i}\left( \overline{a}_{i}+\Delta
a_{i},\overline{b}_{i}+\Delta b_{i}\right) :\lambda \in \mathcal{B},\Vert
(\Delta a_{i},\Delta b_{i})\Vert \leq r_{i},\forall i\right\} }\medskip \\
& =\overline{\mathop{\rm cone}\left\{ \displaystyle\sum\limits_{i=1}^{m}%
\lambda _{i}\left( (\overline{a}_{i},\overline{b}_{i})+r_{i}\mathbb{B}%
_{n+1}\right) :\lambda \in \mathcal{B}\right\} }\medskip .%
\end{array}%
\end{equation*}%
Thus, the conclusion follows.

\ [Proof of \textrm{(b)}] Let $\mathcal{B}$ be the compact base of $K^{\ast
} $. Then, $0_{m}\notin \mathcal{B}$, and so,
\begin{equation*}
\mu :=\min \{\min_{1\leq i\leq m}\big\vert {\lambda }_{i}\big\vert :\lambda
\in \mathcal{B}\}>0.
\end{equation*}%
Define $M=\max \{\Vert \displaystyle\sum\limits_{i=1}^{m}{\lambda }_{i}(%
\overline{a}_{i},\overline{b}_{i})\Vert :\lambda \in \mathcal{B}\}<+\infty .$
We shall prove by contradiction that $C(\overline{A},\overline{b})\subseteq
\frac{M}{\mu }\mathbb{B}_{m}.$ If, in the contrary, $C(\overline{A},%
\overline{b})\nsubseteqq \frac{M}{\mu }\mathbb{B}_{m}$, then we can take $%
r\in C(\overline{A},\overline{b})$ such that $\left\Vert r\right\Vert \geq
\frac{M+\epsilon }{\mu }$ for some $\epsilon >0.$ Now, fix any $\widetilde{%
\lambda }\in \mathcal{B}$. Note that
\begin{equation*}
\displaystyle\sum\limits_{i=1}^{m}\left\vert \widetilde{\lambda }%
_{i}\right\vert r_{i}\geq \mu \displaystyle\sum\limits_{i=1}^{m}r_{i}\geq
\mu \sqrt{\displaystyle\sum\limits_{i=1}^{m}r_{i}^{2}}\geq M+\epsilon \geq
\Vert \displaystyle\sum\limits_{i=1}^{m}\widetilde{\lambda }_{i}(\overline{a}%
_{i},\overline{b}_{i})\Vert +\epsilon .
\end{equation*}%
It follows that
\begin{eqnarray*}
\epsilon \mathbb{B}_{n+1} \subseteq \displaystyle\sum\limits_{i=1}^{m}%
\widetilde{\lambda }_{i}(\overline{a}_{i},\overline{b}_{i})+\displaystyle%
\sum\limits_{i=1}^{m}\left\vert \widetilde{\lambda }_{i}\right\vert r_{i}%
\mathbb{B}_{n+1} & = & \displaystyle\sum\limits_{i=1}^{m}\widetilde{\lambda }%
_{i}(\left( \overline{a}_{i},\overline{b}_{i})+r_{i}\mathbb{B}_{n+1}\right) \\
&\subseteq & \displaystyle\bigcup\limits_{\lambda \in \mathcal{B}}\left\{ %
\displaystyle\sum\limits_{i=1}^{m}\lambda _{i}(\left( \overline{a}_{i},%
\overline{b}_{i})+r_{i}\mathbb{B}_{n+1}\right) \right\} ,
\end{eqnarray*}%
and so
\begin{equation*}
\left( 0_{n},1\right) \in \mathop{\rm cone}\left\{ \displaystyle%
\sum\limits_{i=1}^{m}\lambda _{i}\left( (\overline{a}_{i},\overline{b}%
_{i})+r_{i}\mathbb{B}_{n+1}\right) :\lambda \in \mathcal{B}\right\} .
\end{equation*}%
Then, by Proposition \ref{Prop_Consist_Set3} part (a) , $r\notin C(\overline{%
A},\overline{b})$ (contradiction). Hence, $C(\overline{A},\overline{b})$ is
bounded. Thus, the conclusion follows by the fact that $C(\overline{A},%
\overline{b})$ is radiant. %
%
%

[Proof of \textrm{(c)}] Let us show the ``$\left[ \Longrightarrow \right]$''
direction first. Assume that $\sigma _{r}^{\mathcal{B}}$ satisfies the
Slater condition. As $r\in \mathbb{R}^m_{++}$, there exists $\xi>0$ such
that $r+\xi \mathbb{B}_{m}\subseteq \mathbb{R}_{++}^{m}$. Since $\mathcal{B}$
is a compact base (and so, $0_m \notin \mathcal{B}$), $\eta
:=\max\limits_{\lambda \in \mathcal{B}}\max\limits_{1\leq i\leq m}\left\vert
\lambda _{i}\right\vert $ is a positive real number. Consider the mapping
\begin{equation*}
\begin{array}{ccccc}
&  & (m) &  &  \\
\Phi : & \mathbb{R}^{m}\times & \overbrace{\mathbb{R}^{n+1}\times ...\times
\mathbb{R}^{n+1}} & \longrightarrow & \mathbb{R}^{n+1} \\
& \lambda & \left( z^{1},...,z^{m}\right) &  & \displaystyle %
\sum\limits_{i=1}^{m}\lambda _{i}z^{i}%
\end{array}%
\end{equation*}%
Since $\Phi $\ is continuous (as it has quadratic components) and the set $%
D_{r}:=\mathcal{B\times }\prod\limits_{i=1}^{m}\left[ \left( \overline{a}%
_{i},\overline{b}_{i}\right) +r_{i}\mathbb{B}_{n+1}\right] $ is compact, $%
C_{r}:=\Phi \left( D_{r}\right) $ is a compact subset of $\mathbb{R}^{n+1}$
too. By the Slater condition, there exists $x^{0}\in \mathbb{R}^{n}$ such
that $\left\langle \left( y,y_{n+1}\right) ,\left( x^{0},1\right)
\right\rangle <0$ for all $\left( y,y_{n+1}\right) \in C_{r}.$ Let $\epsilon
>0$ be such that%
\begin{equation*}
\begin{array}{ll}
-\epsilon & =\max \left\{ \left\langle \left( y,y_{n+1}\right) ,\left(
x^{0},1\right) \right\rangle :\left( y,y_{n+1}\right) \in C_{r}\right\} \\
& =\max \left\{ \left\langle \Phi \left( d\right) ,\left( x^{0},1\right)
\right\rangle :d\in D_{r}\right\} .%
\end{array}%
\end{equation*}%
Given $t\in \mathbb{R}_{++}^{m},$ we consider an element of $D_{t}$ of the
form
\begin{equation*}
d^{t}=\left( \lambda ^{t},\left( \overline{a}_{1},\overline{b}_{1}\right)
+t_{1}u^{1,t},...,\left( \overline{a}_{m},\overline{b}_{m}\right)
+t_{m}u^{m,t}\right) ,
\end{equation*}%
with $\lambda ^{t}\in \mathcal{B}$ and $u^{i,t}\in \mathbb{B}_{n+1},$ $%
i=1,...,m,$ arbitrarily chosen. We define the vector $d^{r}:=\left( \lambda
^{t},\left( \overline{a}_{1},\overline{b}_{1}\right)
+r_{1}u^{1,t},...,\left( \overline{a}_{m},\overline{b}_{m}\right)
+r_{m}u^{m,t}\right) \in D_{r}.$ Since%
\begin{equation*}
\left\Vert \Phi \left( d^{t}\right) -\Phi \left( d^{r}\right) \right\Vert
=\left\Vert \displaystyle \sum\limits_{i=1}^{m}\lambda _{i}\left(
t_{i}-r_{i}\right) u^{i,t}\right\Vert \leq m\eta \left\Vert t-r\right\Vert ,
\end{equation*}%
\begin{equation*}
\left\vert \left\langle \Phi \left( d^{t}\right) -\Phi \left( d^{r}\right)
,\left( x^{0},1\right) \right\rangle \right\vert \leq m\eta \left\Vert
\left( x^{0},1\right) \right\Vert \left\Vert t-r\right\Vert ,
\end{equation*}%
and so $\left\langle \Phi \left( d^{t}\right) ,\left( x^{0},1\right)
\right\rangle \leq m\eta \left\Vert \left( x^{0},1\right) \right\Vert
\left\Vert t-r\right\Vert -\epsilon$. Hence, if $$\left\Vert t-r\right\Vert
\leq \min \left\{ \frac{\epsilon }{2m\eta \left\Vert \left( x^{0},1\right)
\right\Vert },\xi \right\} ,$$
then, one has
\begin{equation*}
\max \left\{ \left\langle \left( y,y_{n+1}\right) ,\left( x^{0},1\right)
\right\rangle :\left( y,y_{n+1}\right) \in C_{t}\right\} \leq -\frac{%
\epsilon }{2}<0,
\end{equation*}%
which shows that $x^{0}$ is a Slater point for $\sigma _{t}^{\mathcal{B}}.$
This implies that $t\in C(\overline{A},\overline{b}).$ Thus, one has $r\in %
\mathop{\rm int}C(\overline{A},\overline{b}).$\newline

We now show the reverse direction \textquotedblleft $\left[ \Longleftarrow %
\right] $\textquotedblright . Let $r\in \mathop{\rm int}C(\overline{A},%
\overline{b})$ and for all $\lambda \in \mathcal{B}$, $\displaystyle\sum\limits_{i=1}^{m}r_{i}\lambda _{i}>0.$
 Then there exists $\epsilon >0$ such that
$r+\epsilon 1_{m}\in C(\overline{A},\overline{b})$ and $$\gamma :=\min
\left\{ \displaystyle\sum\limits_{i=1}^{m}r_{i}\lambda _{i}:\lambda \in
\mathcal{B}\right\} >0.$$
We consider perturbations of the coefficients of $\sigma _{r}^{\mathcal{B}}$
preserving the index set $\mathcal{B\times }\displaystyle\prod%
\limits_{i=1}^{m}\left( r_{i}\mathbb{B}_{n+1}\right) $ and we measure such
perturbations by means of the Chebyshev metric $d_{\infty }$. The
coefficients vector of $\sigma _{r}^{\mathcal{B}}$ can be expressed as
\begin{equation}
\displaystyle\sum\limits_{i=1}^{m}\lambda _{i}\left[ \left( \overline{a}_{i},%
\overline{b}_{i}\right) +r_{i}u^{i}\right] =\displaystyle\sum%
\limits_{i=1}^{m}\lambda _{i}\left( \overline{a}_{i},\overline{b}_{i}\right)
+\displaystyle\sum\limits_{i=1}^{m}\lambda _{i}r_{i}u^{i},  \label{2.5}
\end{equation}%
where $\lambda \in \mathcal{B}$ and $u^{i}\in \mathbb{B}_{n+1},$ $i=1,...,m.$
Consider an additive perturbation $\delta w,$ with $0<\delta <\gamma
\epsilon $ and $w\in \mathbb{B}_{n+1}$\ of the coefficient vector in (\ref%
{2.5}). Let $v:=\frac{w}{\displaystyle\sum\nolimits_{1\leq i\leq
m}r_{i}\lambda _{i}}.$ Since $\left\Vert u^{i}+\delta v\right\Vert \leq
1+\delta \left\Vert v\right\Vert \leq 1+\frac{\delta }{\gamma }$ for all $i,$
one has%
\begin{equation*}
\begin{array}{ll}
\displaystyle\sum\limits_{i=1}^{m}\lambda _{i}\left[ \left( \overline{a}_{i},%
\overline{b}_{i}\right) +r_{i}u^{i}\right] +\delta w & =\displaystyle%
\sum\limits_{i=1}^{m}\lambda _{i}\left( \overline{a}_{i},\overline{b}%
_{i}\right) +\displaystyle\sum\limits_{i=1}^{m}\lambda _{i}r_{i}\left(
u^{i}+\delta v\right) \\
& \in \displaystyle\sum\limits_{i=1}^{m}\lambda _{i}\left( \overline{a}_{i},%
\overline{b}_{i}\right) +\displaystyle\sum\limits_{i=1}^{m}\lambda
_{i}r_{i}\left( 1+\frac{\delta }{\gamma }\right) \mathbb{B}_{n+1} \\
& \subseteq \displaystyle\sum\limits_{i=1}^{m}\lambda _{i}\left( \overline{a}%
_{i},\overline{b}_{i}\right) +\displaystyle\sum\limits_{i=1}^{m}\lambda
_{i}r_{i}\left( 1+\epsilon \right) \mathbb{B}_{n+1}.%
\end{array}%
\end{equation*}%
So, the solution set of any perturbed system obtained from $\sigma _{r}^{%
\mathcal{B}}$ by summing up vectors of Euclidean norm less than $\gamma
\epsilon $ to each the coefficient vector contains $F_{r}(\overline{A},%
\overline{b})\neq \emptyset .$ In particular, summing up vectors of
Chebyshev norm less than $\frac{\gamma \epsilon }{\sqrt{m}}$ to each
coefficient vector of $\sigma _{r}^{\mathcal{B}}$ we get a feasible
perturbed system. Hence, by \cite[Theorem 6.1]{GL98}, $\sigma _{r}^{\mathcal{%
B}}$ has a strong Slater solution, which shows that $\sigma _{r}^{\mathcal{B}%
}$ satisfies the Slater condition.
\end{proof}
\begin{acknowledgements}
This research was partially supported by the Australian Research Council,
Discovery Project DP120100467 and the Ministry of Science, Innovation and
Universities of Spain and the European Regional Development Fund (ERDF) of
the European Commission, Grant PGC2018-097960-B-C22.
\end{acknowledgements}
%


\end{singlespace}

\end{document}